\documentclass[11 pt]{amsart}
\usepackage{amscd,amsfonts,amssymb,amsmath}
\usepackage{hyperref}
\usepackage{epsfig}
\usepackage{mathtools}
\usepackage{tabu}
\usepackage{tikz-cd}
\usepackage{physics}
\newtheorem{theorem}{Theorem}[section]
\newtheorem{corollary}[theorem]{Corollary}
\newtheorem{lemma}[theorem]{Lemma}
\newtheorem{proposition}[theorem]{Proposition}
\theoremstyle{definition}

\newtheorem{remark}[theorem]{Remark}
\numberwithin{equation}{subsection}
\newtheorem*{ack}{Acknowledgement}
\usepackage[all,cmtip]{xy}

\usepackage{graphicx} % need this package
\newcommand{\IA}{\operatorname{IA}}

\newcommand{\Aut}{\operatorname{Aut}}

\newcommand{\Ker}{\operatorname{Ker}}
\newcommand{\im}{\operatorname{Im}}

\newcommand{\Inn}{\operatorname{Inn}}

\newcommand{\M}{\operatorname{M}}

\newcommand{\GL}{\operatorname{GL}}
\newcommand{\Z}{\operatorname{Z}}

\newcommand{\id}{\mathrm{id}}

\begin{document}

\title{Structure and automorphisms of pure virtual twin groups}

\author{Tushar Kanta Naik}
\address{School of Mathematical Sciences, National Institute of Science Education and Research, Bhubaneswar, HBNI, P.O. Jatni, Khurda, Odisha 752050, India.}
\email{tushar@niser.ac.in}

\author{Neha Nanda}
\address{Laboratoire de Math\'{e}matiques Nicolas Oresme UMR CNRS 6139, Universit\'{e} de Caen Normandie, 14000 Caen, France.}
\email{nehananda94@gmail.com}

\author{Mahender Singh}
\address{Department of Mathematical Sciences, Indian Institute of Science Education and Research (IISER) Mohali, Sector 81,  S. A. S. Nagar, P. O. Manauli, Punjab 140306, India.}
\email{mahender@iisermohali.ac.in}
\date{\today}


\begin{abstract}
Study of stable isotopy classes of a finite collection of immersed circles without triple or higher intersections on closed oriented surfaces is considered as a planar analogue of virtual knot theory, a far reaching generalisation of classical knot theory. Recent works have established Alexander and Markov theorems in the planar setting. In the classical case, the role of groups is played by twin groups, a class of right-angled Coxeter groups. A new class of groups called virtual twin groups, that extends twin groups in a natural way, plays the role of groups in the virtual case. The virtual twin group $VT_n$ contains the pure virtual twin group $PVT_n$, a planar analogue of the pure Artin braid group. In this paper, we prove that the pure virtual twin group $PVT_n$ is an irreducible right-angled Artin group with trivial center and give it's precise presentation. We show that $PVT_n$ has a decomposition as an iterated semi-direct product of infinite rank free groups. We give a complete description of the automorphism group of $PVT_n$ and establish splitting of natural exact sequences of automorphism groups. As applications, we show that $VT_n$ is residually finite and $PVT_n$ has the $R_\infty$-property. 
\end{abstract}

\maketitle

\section{Introduction}
Doodles on a 2-sphere were introduced by Fenn and Taylor \cite{FennTaylor} as a finite collection of simple closed curves without triple or higher intersections on a  2-sphere. Allowing self intersection of curves, Khovanov \cite{Khovanov} extended the idea to a finite collection of closed curves without triple or higher intersections on a closed oriented surface. He also introduced an analogue of the link group for doodles and constructed several infinite families of doodles whose fundamental groups have infinite center. Recently, Bartholomew-Fenn-Kamada-Kamada \cite{BartholomewFennKamada2018, BartholomewFennKamada2018-2} extended the study of doodles to immersed circles on closed oriented surfaces of any genus, which can be thought of as planar analogue of virtual knot theory. It is a natural problem to look for invariants that could classify these geometric objects. Coloring of diagrams using a special type of algebra has been used to construct an invariant for virtual doodles \cite{BartholomewFennKamada2019}. Further, an Alexander type invariant for oriented doodles which vanishes on unlinked doodles with more than one component has been constructed in a recent work \cite{CisnerosFloresJuyumayaMarquez}.
\par

Analogous to classical knot theory, the theory of doodles is modelled on an appropriate class of groups. The role of groups for doodles on a  2-sphere is played by so called {\it twin groups}. The twin groups $T_n$, $n \ge 2$, form a special class of right-angled Coxeter groups and first appeared in the work of Shabat and Voevodsky \cite{ShabatVoevodsky}, wherein they were referred as {\it Grothendieck cartographical groups}. Later, these groups were brought to light by Khovanov \cite{Khovanov} under the present name, who also gave a geometric interpretation of these groups similar to the one for classical braid groups. Consider configurations of $n$ arcs in the infinite strip $\mathbb{R} \times  [0,1]$ connecting $n$ marked points on each of the parallel lines $\mathbb{R} \times \{1\}$ and $\mathbb{R} \times \{0\}$ such that each arc is monotonic and no three arcs have a point in common. Two such configurations are equivalent if one can be deformed into the other by a homotopy of such configurations in $\mathbb{R} \times [0,1]$ keeping the end points of arcs fixed. An equivalence class under this equivalence is called a {\it twin}. The product of two twins is defined by placing one twin on top of the other, and the collection of all twins with $n$ arcs under this operation forms a group that is isomorphic to $T_n$. Taking the one point compactification of the plane, one can define the closure of a twin on a $2$-sphere analogous to the closure of a geometric braid in $\mathbb{R}^3$. Khovanov also proved that every oriented doodle on a $2$-sphere is a closure of a twin, which is an analogue of classical Alexander Theorem. An analogue of Markov Theorem for doodles on a 2-sphere has been established recently by Gotin \cite{Gotin}.
\par

The pure twin group $PT_n$ is the kernel of the natural surjection from $T_n$ onto the symmetric group $S_n$ on $n$ symbols, and can be thought of as a planar analogue of the pure braid group. In a recent work \cite{BarVesSin}, Bardakov-Singh-Vesnin proved that $PT_n$ is free for $n = 3,4$ and not free for $n \geq 6$. Gonz\'alez-Le\'on-Medina-Roque \cite{GonGutiRoq} showed that $PT_5$ is a free group of rank $31$. A precise description of $PT_6$ has been obtained recently by Mostovoy and Roque-M\'arquez \cite{MostRoq}, wherein they proved that $PT_6$ is isomorphic to the free product of $F_{71}$ and 20 copies of  $\mathbb{Z} \oplus \mathbb{Z}$. Finally, a minimal presentation of $PT_n$ for all $n$ has been announced in a recent preprint by Mostovoy \cite{Mostovoy}. Automorphisms, (twisted) conjugacy classes and centralisers of involutions in twin groups have been explored in recent works of the authors \cite{NaikNandaSingh1, NaikNandaSingh2}. It is worth noting that (pure) twin groups have been used by physicists  \cite{HarshmanKnapp} who refer them as (pure) traid groups.
\par

It is quite reasonable to think of the study of stable isotopy classes of immersed circles without triple or higher intersection points on closed oriented surfaces as a planar analogue of virtual knot theory where the genus zero case corresponds to classical knot theory. Recall that virtual knot theory is a far reaching generalisation of classical knot theory and pioneered by Kauffman in \cite{Kauffman1999}. As noted before Alexander and Markov theorems for the genus zero case are already known with the role of groups being played by twin groups. A recent work \cite{NandaSingh2020} by Nanda and Singh proves Alexander and Markov theorems for the higher genus case. It is proved that virtual twin groups, denoted $VT_n$, introduced in \cite{BarVesSin} as abstract generalisation of twin groups, play the role of groups in the theory of virtual doodles. These correspondences can be summarised as
$$\bigcup_{n \ge 2} T_n/_{\textrm{Markov equivalence}}\quad  \longleftrightarrow \quad\textrm{Homotopy classes of doodles on 2-sphere}$$
and
$$\bigcup_{n \ge 2} VT_n/_{\textrm{Markov equivalence}} \quad\longleftrightarrow \quad\textrm{Stable equivalence classes of doodles on surfaces}.$$
\par
The virtual twin group $VT_n$ extends the twin group $T_n$ and surjects onto the symmetric group $S_n$ in a natural way. A pure analogue of the virtual twin group called the pure virtual twin group, denoted $PVT_n$, is defined as the kernel of the natural surjection from $VT_n$  onto $S_n$.
\par

The purpose of this paper is to develop the theory of (pure) virtual twin groups. We investigate their important structural aspects in detail. We begin by recalling the definition and the geometrical interpretation of virtual twin groups in Section \ref{section-prelim}. In Section \ref{section-presentation-pvtn}, we obtain a presentation of the pure virtual twin group $PVT_n$ (Theorem \ref{pvtn-right-angled-artin}), which quite interestingly turns out to be an irreducible right-angled Artin group. As an application, we deduce that $VT_n$ is residually finite (Corollary \ref{vtn-residually-finite}). In Section \ref{section-decom-semidirect-product}, we prove that $PVT_n$ can be written as an iterated semi-direct product of infinite rank free groups (Theorem \ref{semidirect-decomposition-ptn}), which is an analogue of a similar result for pure braid groups \cite{Markoff1945} and pure virtual braid groups \cite{Bardakov2004}. As a consequence, we show that $PVT_n$, and hence $VT_n$ has trivial center (Corollary \ref{center-of-PVT_n}). Section \ref{section-auto-groupo-pvtn} is devoted to the study of automorphisms of $PVT_n$. We give a complete description of $\Aut(PVT_n)$ for $n \ge 5$ in Theorem \ref{full-auto-decomposition} and also establish splitting of natural short exact sequences of automorphism groups (Theorem \ref{aut-pvtn-main}). The case $n=4$ is exotic and a complete description of $\Aut(PVT_4)$ is given in Theorem \ref{splitting-auto-pvt4}. Understanding of automorphisms of $PVT_n$ help us to conclude that   $PVT_n$ has $R_\infty$-property if and only if $n \ge 3$ (Theorem \ref{pvtn-r-infinity}). We also deduce that each $\IA$ automorphism of $PVT_n$ is inner if and only if $n=2$ or $n \ge 5$ (Corollary \ref{cor-ia-auto}).
\medskip

\section{Preliminaries}\label{section-prelim}
The {\it virtual twin group} $VT_n$, $n \ge 2$, is a generalisation of the twin group $T_n$. The group $VT_n$ has generators $\{ s_1, s_2, \ldots, s_{n-1}, \rho_1, \rho_2, \ldots, \rho_{n-1}\}$ and defining relations
\begin{eqnarray}
s_i^{2} &=&1 \hspace*{5mm} \textrm{for } i = 1, 2, \dots, n-1, \label{1}\\ 
s_is_j &=& s_js_i \hspace*{5mm} \textrm{for } |i - j| \geq 2,\label{2}\\
\rho_i^{2} &=& 1 \hspace*{5mm} \textrm{for } i = 1, 2, \dots, n-1, \label{3}\\
\rho_i\rho_j &=& \rho_j\rho_i \hspace*{5mm} \textrm{for } |i - j| \geq 2, \label{4}\\
\rho_i\rho_{i+1}\rho_i &=& \rho_{i+1}\rho_i\rho_{i+1}\hspace*{5mm} \textrm{for } i = 1, 2, \dots, n-2, \label{5}\\
\rho_is_j &=& s_j\rho_i \hspace*{5mm} \textrm{for } |i - j| \geq 2, \label{6}\\
\rho_i\rho_{i+1} s_i &=& s_{i+1} \rho_i \rho_{i+1}\hspace*{5mm} \textrm{for } i = 1, 2, \dots, n-2. \label{7}
\end{eqnarray}
Note that the twin group $T_n$ has a presentation
$$\big\langle s_1, s_2, \dots, s_{n-1}~|~ s_i^{2} = 1~ \text{for}~1 \le i \le n-1~\textrm{and}~ s_is_j = s_js_i~ \text{for}~ |i-j| \geq 2 \big\rangle.$$
Analogous to (virtual) braid groups \cite{KauffmanLambropoulou2004}, virtual twin groups have a nice geometrical interpretation (see \cite{NandaSingh2020} for more details). We consider a set $Q$ of fixed $n$ points on the real line $\mathbb{R}$. Then a \textit{virtual twin diagram} on $n$ strands is a subset $D$ of the strip $\mathbb{R} \times [0,1]$ consisting of $n$ intervals called {\it strands} such that the boundary of $D$ is $Q  \times \{0,1\}$ and the following conditions are satisfied:
\begin{enumerate}
\item the natural projection $\mathbb{R} \times [0,1] \to [0,1]$ maps each strand homeomorphically onto $[0,1]$, 
\item the set $V(D)$ of all crossings of the diagram $D$ consists of transverse double points of $D$ where each crossing has the pre-assigned information of being a real or a virtual crossing as depicted in Figure \ref{Crossings}. A virtual crossing is depicted by a crossing encircled with a small circle.
\end{enumerate}
\begin{figure}[hbtp]
\centering
\includegraphics[scale=0.4]{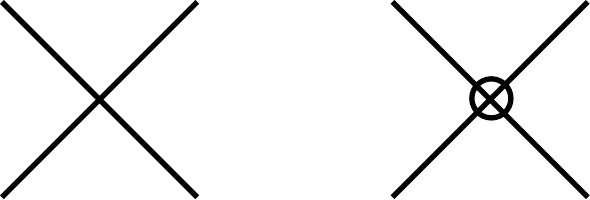}
\caption{Real and virtual crossings}
\label{Crossings}
\end{figure}

We say that two virtual twin diagrams $D_1$ and $D_2$ on $n$ strands are \textit{equivalent} if one can be obtained from the other by a finite sequence of planar Reidemeister moves as shown in Figure \ref{ReidemeisterMoves} and isotopies of the plane. 
\begin{figure}[hbtp]
\centering
\includegraphics[scale=0.35]{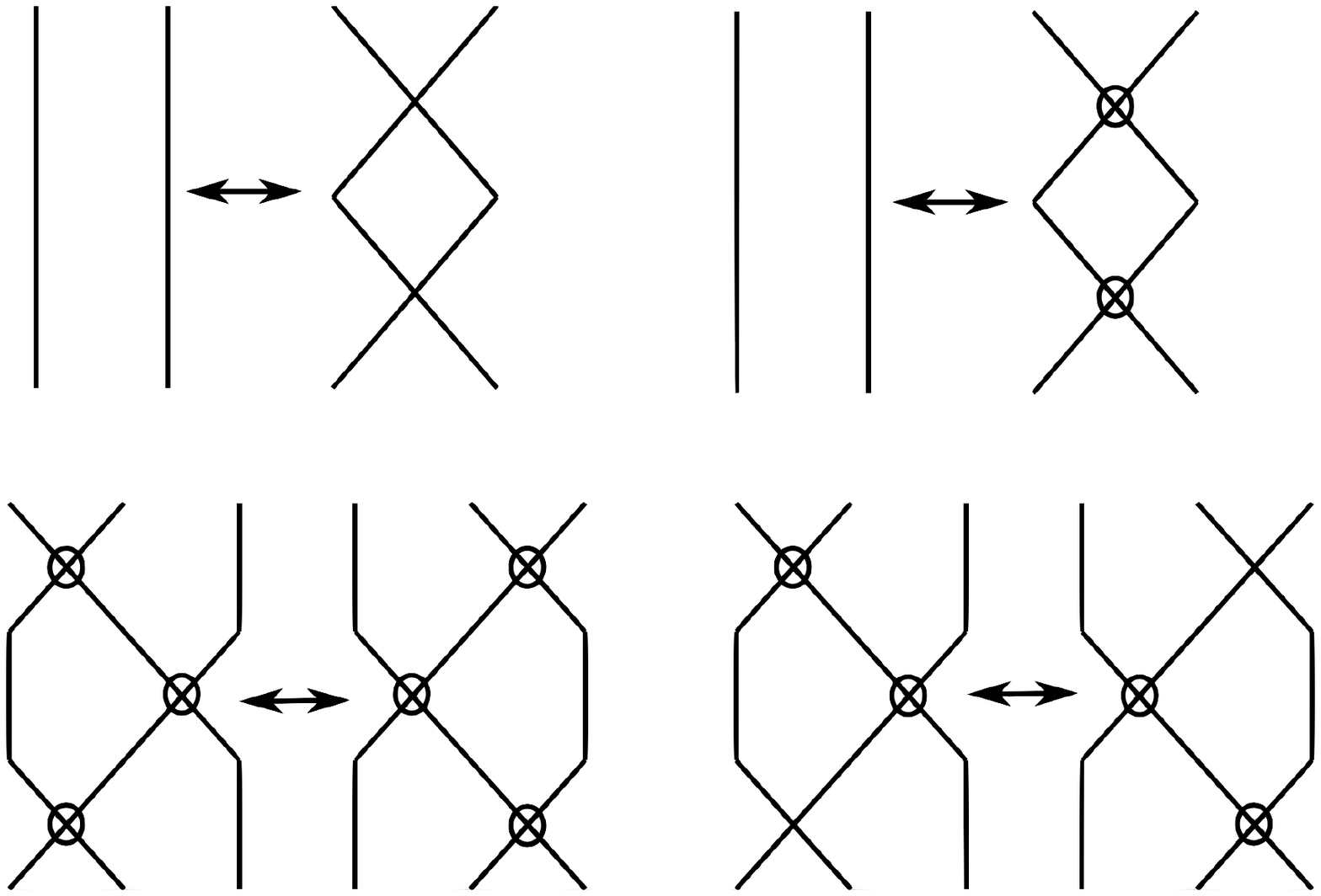}
\caption{Reidemeister moves for virtual twin diagrams}
\label{ReidemeisterMoves}
\end{figure}

A \textit{virtual twin} is then defined as an equivalence class of such virtual twin diagrams. The product $D_1D_2$ of two virtual twin diagrams $D_1$ and $D_2$ is defined by placing $D_1$ on top of $D_2$ and then shrinking the interval to $[0,1]$. It is clear that if  $D_1$ is equivalent to $D_1'$ and $D_2$ is equivalent to $D_2'$, then $D_1D_2$ is equivalent to $D_1'D_2'$. Thus, there is a well-defined binary operation on the set of all virtual twins on $n$ strands. It has been shown in \cite[Proposition 3.3]{NandaSingh2020} that the set of all virtual twins on $n$ strands is a group that is isomorphic to the abstractly defined group $VT_n$. The generators $s_i$ and $\rho_i$ of $VT_n$ can be geometrically represented as in Figure \ref{generator-vtn}.
\begin{figure}[hbtp]
\centering
\includegraphics[scale=0.3]{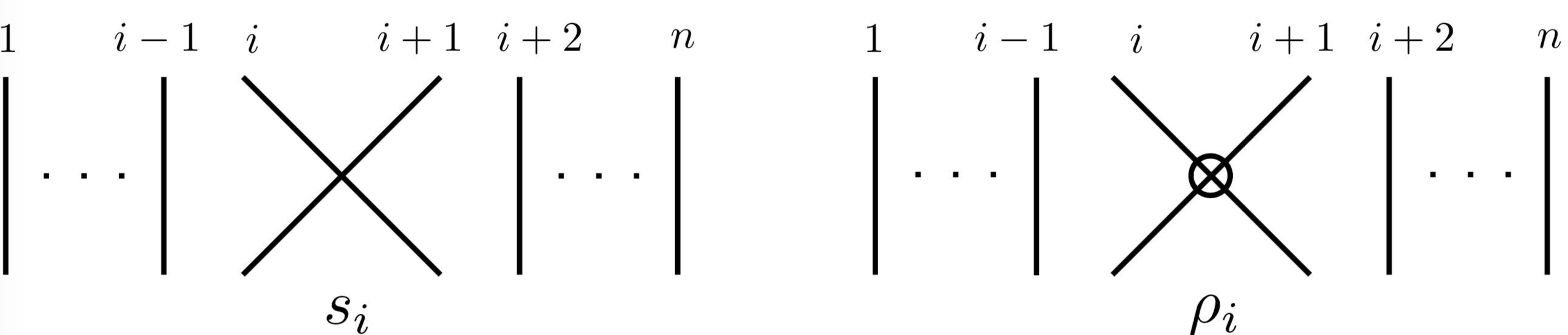}
\caption{Generator $s_i$ and $\rho_i$}
\label{generator-vtn}
\end{figure}

Note that $VT_2 \cong \mathbb{Z}_2 * \mathbb{Z}_2$, the infinite dihedral group. There is a natural surjection $\pi:VT_n  \to S_n$ given by
$$\pi(s_i) = \pi(\rho_i) = (i, i+1)$$
for all $1\leq i \leq n-1$. The kernel $PVT_n$ of this surjection is called the \textit{pure virtual twin group} on $n$ strands. The group $PVT_n$ is an analogue of the pure virtual braid group. The map $S_n \to VT_n$ given by $(i, i+1)\mapsto \rho_i$ is a splitting of the short exact sequence
$$1 \to PVT_n \to VT_n \to S_n \to 1,$$
and hence $VT_n= PVT_n \rtimes S_n$. Thus, we can identify the subgroup $\langle \rho_1, \rho_2, \ldots,\rho_{n-1}\rangle$ of $VT_n$ with the symmetric group $S_n$ on $n$ symbols. 
\par
We conclude this section by setting some notation. For elements $g, h$ of a group $G$, we denote the commutator $g^{-1}h^{-1}gh$ by $[ g, h ]$ and the element $h^{-1}gh$ by $g^h$.
\medskip

\section{Presentation of pure virtual twin group $PVT_n$}\label{section-presentation-pvtn}
In this section, we give a presentation of $PVT_n$. We show that the rank of $PVT_n$ is $n(n-1)/2$, which, interestingly, coincides with the rank of the pure braid group.
\medskip

We shall use the standard presentation of $VT_n$ and the Reidemeister-Schreier method \cite[Theorem 2.6]{Magnus1966}. For each $1 \leq k \leq n-1$, set 
\begin{equation*}
m_{k, i_k}:=
\begin{cases}
\rho_k \rho_{k-1}\dots \rho_{i_{k}+1} & \text{ for }  0 \leq i_k < k, \\
1 & \text{ for }    i_k = k, \\
\end{cases}
\end{equation*}
and 
$$\M_n := \big\{ m_{1, i_1}m_{2, i_2}\dots m_{n-1, i_{n-1}} ~| ~0 \leq i_k \leq k  \text{ for each } 1 \leq k \leq n-1   \big\}$$
as the Schreier system of coset representatives of $PVT_n$ in $VT_n$. For an element $w \in VT_n$, let $\overline{w}$ denote the unique coset representative of the coset of $w$ in the Schreier set $\M_n$. By Reidemeister-Schreier method, the group $PVT_n$ is generated by the set
$$\big\{ \gamma( \mu, a) = (\mu a) (\overline{\mu a})^{-1} ~|~ \mu \in \M_n \text{ and } a \in \{s_1, \dots, s_{n-1}, \rho_1, \dots, \rho_{n-1}\} \big\}$$
with defining relations
$$\big\{\tau(\mu r \mu^{-1}) \mid \mu \in \M_n~\text{and}~ r ~\text{is a defining relation in} ~VT_n \big\},$$
where $\tau$ is the rewriting process. More precisely, for an element $g=g_1g_2\dots g_k \in VT_n$, we have 
$$\tau(g) = \gamma(1, g_1)\gamma(\overline{g_1}, g_2)\dots \gamma(\overline{g_1g_2\dots g_{k-1}}, g_k).$$

We set
$$\lambda_{i, i+1}= s_i \rho_i$$
for each $1 \le i \le n-1$ and
$$\lambda_{i,j} = \rho_{j-1} \rho_{j-2} \dots \rho_{i+1} \lambda_{i, i+1} \rho_{i+1} \dots \rho_{j-2}  \rho_{j-1}$$
for each $1 \leq i < j \leq n$ and $j \ne i+1$. These notations will be used throughout this section.
\medskip

\begin{theorem}\label{PVTn-generators}
The pure virtual twin group $PVT_n$ on  $n \ge 2 $ strands is generated by $$\mathcal{S}= \big\{ \lambda_{i, j} ~|~  1 \leq i< j \leq n \big\}.$$
\end{theorem}

\begin{proof}
The case $n=2$ is immediate, and hence we assume $n \ge 3$. Recall that $PVT_n$ is generated by the elements $\gamma( \mu, a)$, where $\mu \in \M_n$ and $a \in \{s_1, \dots, s_{n-1}, \rho_1, \dots, \rho_{n-1}\}$.
We observe that 
\begin{eqnarray*}
\overline{\alpha} &=& \alpha,\\
\overline{ \alpha_1 s_{i_1} \dots \alpha_k s_{i_k}} &=& \alpha_1 \rho_{i_1} \dots \alpha_k \rho_{i_k} 
\end{eqnarray*}
  in $VT_n$ for elements $\alpha, \alpha_j$ in the subgroup $\langle \rho_1, \dots, \rho_{n-1}\rangle$. Therefore, we have 
$$\gamma( \mu, \rho_i)  = (\mu \rho_i) (\mu \rho_i)^{-1} = 1$$
and
$$\gamma( \mu, s_i) = (\mu s_i) (\mu \rho_i)^{-1} = \mu s_i \rho_i \mu^{-1} = \mu \lambda_{i, i+1} \mu^{-1}$$
for each  $\mu \in \M_n$ and $ i=1, 2, \dots, n-1$. Let $\mathcal{S} \sqcup \mathcal{S}^{-1}= \{\lambda_{i,j}^{\pm 1}~|~ \lambda_{i,j} \in \mathcal{S} \}$.  We claim that each $\gamma( \mu, s_i)$ lies in  $\mathcal{S} \sqcup \mathcal{S}^{-1}$. For this, we analyse the conjugation action of $S_n = \langle \rho_1, \ldots, \rho_{n-1} \rangle$ on the set $\mathcal{S}$.
\medskip

First consider $\lambda_{i,i+1}$ for $i=1,2, \dots, n-1$.
\begin{itemize}
\item[(i)] If $ 1 \leq k \leq i-2$ or $i+2 \leq k \leq n-1$, then
\begin{equation*}\label{stabliser12}
 \rho_k \lambda_{i,i+1} \rho_k = \lambda_{i,i+1}.
 \end{equation*}

\item[(ii)] If $ k= i-1$, then
\begin{align*}
 \rho_k \lambda_{i,i+1} \rho_k &= \rho_{i-1} \lambda_{i,i+1} \rho_{i-1}\\
 &= \rho_{i-1} s_i \rho_i \rho_{i-1} \\
&= \rho_{i-1} s_i \rho_{i-1}\rho_{i-1}\rho_i \rho_{i-1} \quad (\text{Using } \eqref{3}) \\		
&= \rho_{i} s_{i-1} \rho_{i}\rho_{i-1}\rho_i \rho_{i-1} \quad \quad (\text{Using } \eqref{7}) \\	
&= \rho_{i} s_{i-1} \rho_{i-1}\rho_{i}\rho_{i-1} \rho_{i-1} \quad (\text{Using } \eqref{5}) \\	
&= \rho_{i}(s_{i-1} \rho_{i-1})\rho_{i} \quad \quad \quad \quad (\text{Using } \eqref{3}) \\
&= \lambda_{i-1,i+1}.
 \end{align*}
 
\item[(iii)] If $k=i$, then
$$ \rho_k \lambda_{i,i+1} \rho_k = \lambda_{i,i+1}^{-1}. $$

\item[(iii)] If $k=i+1$, then
$$ \rho_k \lambda_{i,i+1} \rho_k = \lambda_{i,i+2}. $$
\end{itemize}
\medskip

Next, we consider $\lambda_{i,j}$ for each $1 \leq i < j \leq n$ and $j \ne i+1$. 
\begin{itemize}  

\item[(i)] If $1 \leq k \leq i-2$ or $j+1 \leq k \leq n-1$, then
$$\rho_k \lambda_{i,j}\rho_k = \lambda_{i,j}.$$

\item[(ii)] For $k=i-1$, we have $\rho_{i-1} \lambda_{i,j} \rho_{i-1} = \lambda_{i-1,j}$
since
\begin{align*}
\rho_{i-1} \lambda_{i,j} \rho_{i-1} &= \rho_{i-1} \rho_{j-1} \rho_{j-2} \dots \rho_{i+1} \lambda_{i, i+1} \rho_{i+1} \dots \rho_{j-2}  \rho_{j-1} \rho_{i-1} \\
&= \rho_{j-1} \rho_{j-2} \dots \rho_{i+1}  \rho_{i-1} \lambda_{i, i+1} \rho_{i-1} \rho_{i+1} \dots \rho_{j-2}  \rho_{j-1} \qquad \qquad \qquad \qquad (\text{Using } \eqref{4}) \\
&= \rho_{j-1} \rho_{j-2} \dots \rho_{i+1}  \rho_{i-1} s_i \rho_i \rho_{i-1} \rho_{i+1} \dots \rho_{j-2}  \rho_{j-1}\\
&= \rho_{j-1} \rho_{j-2} \dots \rho_{i+1} \rho_i \rho_i \rho_{i-1}  s_i \rho_i \rho_{i-1}\rho_{i+1} \dots \rho_{j-2}  \rho_{j-1} \qquad \qquad \qquad \quad (\text{Using } \eqref{3}) \\
&= \rho_{j-1} \rho_{j-2} \dots \rho_{i+1} \rho_i s_{i-1} \rho_i \rho_{i-1} \rho_i \rho_{i-1}\rho_{i+1} \dots \rho_{j-2}  \rho_{j-1} \qquad \qquad \qquad(\text{Using } \eqref{7}) \\
&= \rho_{j-1} \rho_{j-2} \dots \rho_{i+1} \rho_i s_{i-1} \rho_i \rho_{i} \rho_{i-1} \rho_{i}\rho_{i+1} \dots \rho_{j-2}  \rho_{j-1} \qquad \qquad \qquad \quad (\text{Using } \eqref{5}) \\
&= \rho_{j-1} \rho_{j-2} \dots \rho_{i+1} \rho_i (s_{i-1} \rho_{i-1}) \rho_{i}\rho_{i+1} \dots \rho_{j-2}  \rho_{j-1} \qquad \qquad \qquad \qquad  (\text{Using } \eqref{3}) \\
&= \lambda_{i-1,j}.
\end{align*}

\item[(iii)] For $k=i$, we have $\rho_{i} \lambda_{i,j} \rho_{i} = \lambda_{i+1,j},$
since
\begin{align*}
\rho_{i} \lambda_{i,j} \rho_{i} &= \rho_{i} \rho_{j-1} \rho_{j-2} \dots \rho_{i+1} \lambda_{i, i+1} \rho_{i+1} \dots \rho_{j-2}  \rho_{j-1} \rho_{i}\\
&= \rho_{j-1} \rho_{j-2} \dots \rho_i \rho_{i+1} \lambda_{i, i+1} \rho_{i+1} \rho_i \dots \rho_{j-2}  \rho_{j-1} \qquad \qquad \qquad \qquad \qquad \qquad (\text{Using } \eqref{4}) \\
&= \rho_{j-1} \rho_{j-2} \dots \rho_i \rho_{i+1} s_i \rho_i \rho_{i+1} \rho_i \dots \rho_{j-2}  \rho_{j-1}\\
&= \rho_{j-1} \rho_{j-2} \dots \rho_{i+2} s_{i+1} \rho_i \rho_{i+1} \rho_i \rho_{i+1} \rho_i \dots \rho_{j-2}  \rho_{j-1} \qquad \qquad \qquad \qquad \qquad (\text{Using } \eqref{7}) \\
&= \rho_{j-1} \rho_{j-2} \dots \rho_{i+2} s_{i+1} \rho_i \rho_{i} \rho_{i+1} \rho_{i} \rho_i \dots \rho_{j-2}  \rho_{j-1} \qquad \qquad \qquad \qquad \qquad \quad (\text{Using } \eqref{5}) \\
&= \rho_{j-1} \rho_{j-2} \dots \rho_{i+2}( s_{i+1} \rho_{i+1}) \rho_{i+2} \dots \rho_{j-2}  \rho_{j-1}\qquad \qquad \qquad \qquad \qquad \quad \enspace (\text{Using } \eqref{3}) \\
&= \lambda_{i+1,j}.
\end{align*}

\item[(iv)] If $ i+1 \leq k \leq j-2$, then 
\begin{align*}
\rho_k \lambda_{i,j} \rho_k &= \rho_k \rho_{j-1} \dots \rho_{k+1} \rho_{k} \dots \rho_{i+1} \lambda_{i, i+1} \rho_{i+1} \dots \rho_k \rho_{k+1} \dots \rho_{j-1} \rho_k \\
&= \rho_{j-1} \dots \rho_k \rho_{k+1} \rho_{k} \dots \rho_{i+1} \lambda_{i, i+1} \rho_{i+1} \dots \rho_k \rho_{k+1} \rho_k  \dots \rho_{j-1}  \qquad \qquad \qquad \quad (\text{Using } \eqref{4}) \\
&= \rho_{j-1} \dots \rho_{k+1} \rho_{k} \rho_{k+1} \rho_{k-1} \dots \rho_{i+1} \lambda_{i, i+1} \rho_{i+1} \dots \rho_{k-1} \rho_{k+1} \rho_{k} \rho_{k+1}  \dots \rho_{j-1} \quad (\text{Using } \eqref{5}) \\
&= \rho_{j-1} \dots \rho_{k+1} \rho_{k} \rho_{k-1} \dots \rho_{i+1} \lambda_{i, i+1} \rho_{i+1} \dots \rho_{k-1}\rho_{k} \rho_{k+1}  \dots \rho_{j-1} \qquad \qquad\quad (\text{Using } \eqref{4}) \\
&= \lambda_{i,j}.
\end{align*}

\item[(v)] If $k= j-1$, then 
$$\rho_k \lambda_{i,j}\rho_k = \lambda_{i,j-1}.$$

\item[(vi)] If $k= j$, then 
$$\rho_k \lambda_{i,j}\rho_k = \lambda_{i,j+1}.$$

\end{itemize}

Hence, each generator $\gamma( \mu, s_i)$ lies in the set $\mathcal{S} \sqcup \mathcal{S}^{-1}$. Conversely, if $\lambda_{i,j} \in \mathcal{S}$ is an arbitrary element, then we see that conjugation by $(\rho_{i-1}\rho_{i-2}\ldots \rho_2\rho_1)(\rho_{j-1}\rho_{j-2}\ldots \rho_3\rho_2)$ maps 
$\lambda_{1,2}$ to $\lambda_{i,j}$, whereas conjugation by $(\rho_{i-1}\rho_{i-2}\ldots \rho_2\rho_1)(\rho_{j-1}\rho_{j-2}\ldots \rho_3\rho_2\rho_1)$ maps $\lambda_{1,2}$ to $\lambda_{i,j}^{-1}$. That is, each $\lambda_{i,j}= \mu \lambda_{1, 2} \mu^{-1}=\gamma(\mu, s_1)$ for some $\mu \in \M_n$, and hence $\mathcal{S}$ generates the group $PVT_n$.
\end{proof}

\begin{remark}\label{rephrasing-action}
We can summarise the action of $S_n$ on the set  $\mathcal{S} \sqcup \mathcal{S}^{-1}$ as
\begin{equation*}
\rho_i :
\begin{cases}
\lambda_{i, i+1} \longleftrightarrow \lambda_{i, i+1}^{-1}, & \\
\lambda_{i, j} \longleftrightarrow \lambda_{i+1, j} & \text{for all}\;\; i+2 \leq j \leq n,\\
\lambda_{j, i} \longleftrightarrow \lambda_{j, i+1} & \text{for all}\;\; 1\leq j < i,\\
\lambda_{k, l} \longleftrightarrow \lambda_{k, l} & \text{otherwise, i.e.,}\;\; k \geq i+2, \text{or}\;\; k < i \;\text{and}\; i\neq l \neq i+1.
\end{cases}
\end{equation*}
The action can be further simplified as
\begin{equation*}
\rho_i :
\begin{cases}
\lambda_{i, i+1} \longleftrightarrow \lambda_{i, i+1}^{-1}, & \\
\lambda_{k, l} \longleftrightarrow \lambda_{k', l'} & \text{for all}\;\; (k, l) \neq (i, i+1),
\end{cases}
\end{equation*} 
where the transposition $(k', l')$ equals $(i, i+1) (k, l) (i, i+1)$ and $k' < l'$. As seen in the proof of Theorem \ref{PVTn-generators}, the action of $S_n$ on the set $\mathcal{S} \sqcup \mathcal{S}^{-1}$ is transitive. 
\end{remark}

\begin{theorem}\label{pvtn-right-angled-artin}
The pure virtual twin group $PVT_n$ on  $n \ge 2 $ strands has the presentation
$$\big\langle \lambda_{i,j},~1 \leq i < j \leq n ~|~ \lambda_{i,j} \lambda_{k,l} =  \lambda_{k,l} \lambda_{i,j} \text{ for distinct integers } i, j, k, l \big\rangle.$$
\end{theorem}

\begin{proof}
Theorem \ref{PVTn-generators} already gives a generating set $\mathcal{S}$ for $PVT_n$. Geometrically, a generator $\lambda_{i,j}$ looks as in Figure \ref{fig-generator-pvtn}.
\begin{figure}[hbtp]
\centering
\includegraphics[scale=1]{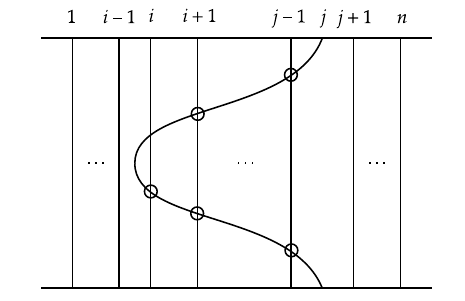}
\caption{The generator $\lambda_{i,j}$}
\label{fig-generator-pvtn}
\end{figure}

The defining relations are given by 
$$\tau(\mu r \mu^{-1}),$$
where $\tau$ is the rewriting process, $\mu \in \M_n$ and $r$ is a defining relation in $VT_n$. Let us take $$ \mu = \rho_{i_1} \rho_{i_2} \dots \rho_{i_k} \in \M_n.$$ 
Note that, since $\gamma(\mu, \rho_i)=1$ for all $i$, no non-trivial relations for $PVT_n$ can be obtained from the relations \eqref{3}--\eqref{5} of $VT_n$. We consider the remaining relations one by one. 
\begin{itemize}
\item[(i)] First consider the relations $s_i^2 = 1$, $1 \le i \le n-1$, of $VT_n$. Then we have
\begin{align*}
\tau(\mu s_i^2 \mu^{-1}) &= \gamma(\rho_{i_1} \dots \rho_{i_k} s_i s_i \rho_{i_k} \dots \rho_{i_1})\\
&= \gamma(1, \rho_{i_1}) \gamma(\overline{\rho_{i_1}}, \rho_{i_2}) \dots \gamma(\overline{\mu}, s_i) \gamma(\overline{\mu s_i}, s_i) \dots \gamma(\overline{\mu s_i s_i \rho_{i_k} \dots \rho_{i_2} }, \rho_{i_1})\\
&= \gamma(\overline{\mu}, s_i) \gamma(\overline{\mu s_i}, s_i)\\
&= \gamma(\mu, s_i) \gamma(\mu \rho_i, s_i)\\
&= (\mu s_i \rho_i \mu^{-1}) (\mu \rho_i s_i \mu^{-1})\\
&= (\mu s_i \rho_i \mu^{-1}) (\mu s_i \rho_i \mu^{-1})^{-1},
\end{align*}
which does not yield any non-trivial relation in $PVT_n$.
\item[]

\item[(ii)] Next we consider the relations $ (s_i \rho_j)^2 = 1$ for $|i-j| \geq 2$. Then we have
\begin{align*}
\tau(\mu s_i \rho_j s_i \rho_j \mu^{-1}) &= \gamma(\overline{\mu}, s_i) \gamma(\overline{\mu s_i}, \rho_j)  \gamma(\overline{\mu s_i \rho_j}, s_i)  \gamma(\overline{ \mu s_i \rho_j s_i}, \rho_j)\\
&= \gamma(\overline{\mu}, s_i) \gamma(\overline{\mu s_i \rho_j}, s_i)\\
&= \gamma(\mu, s_i) \gamma(\mu \rho_i \rho_j, s_i)\\
&= (\mu s_i \rho_i \mu^{-1}) (\mu \rho_i \rho_j s_i \rho_i \rho_j \rho_i \mu^{-1})\\
&= (\mu s_i \rho_i \mu^{-1}) (\mu \rho_i s_i \mu^{-1})\\
&= (\mu s_i \rho_i \mu^{-1})(\mu s_i \rho_i \mu^{-1})^{-1},
\end{align*}
which gives a trivial relation in $PVT_n$.
\item[]
\item[(iii)] Now we consider the relations $\rho_i s_{i+1} \rho_i \rho_{i+1} s_i \rho_{i+1}=1$, where $1 \le i\le n-2$. Computing
\begin{align*}
\tau(\mu \rho_i s_{i+1} \rho_i \rho_{i+1} s_i \rho_{i+1} \mu^{-1})&= \gamma(\overline{\mu \rho_i}, s_{i+1}) \gamma(\overline{ \mu \rho_i s_{i+1} \rho_i \rho_{i+1}}, s_i)\\
&= \gamma(\mu \rho_i, s_{i+1}) \gamma(\mu \rho_i \rho_{i+1} \rho_i \rho_{i+1}, s_i)\\
&= (\mu \rho_i s_{i+1} \rho_{i+1} \rho_i \mu^{-1}) (\mu \rho_{i+1} \rho_i s_i \rho_i \rho_i \rho_{i+1} \mu^{-1})\\
&= (\mu \rho_i s_{i+1} \rho_{i+1} \rho_i\mu^{-1}) (\mu \rho_{i+1} \rho_i s_i \rho_{i+1} \mu^{-1})\\
&= (\mu \rho_i s_{i+1} \rho_{i+1} \rho_i\mu^{-1}) (\mu \rho_{i+1} \rho_i  \rho_{i+1}  \rho_{i+1} s_i \rho_{i+1} \mu^{-1}) \qquad (\text{Using } \eqref{3}) \\
&= (\mu \rho_i s_{i+1} \rho_{i+1} \rho_i\mu^{-1}) (\mu \rho_{i+1} \rho_i  \rho_{i+1}  \rho_{i} s_{i+1} \rho_{i} \mu^{-1}) \qquad \quad (\text{Using } \eqref{7}) \\
&= (\mu \rho_i s_{i+1} \rho_{i+1} \rho_i\mu^{-1}) (\mu \rho_{i+1} \rho_{i+1}  \rho_{i}  \rho_{i+1} s_{i+1} \rho_{i} \mu^{-1}) \qquad (\text{Using } \eqref{7}) \\
&= (\mu \rho_i s_{i+1} \rho_{i+1} \rho_i\mu^{-1}) (\mu \rho_{i}  \rho_{i+1} s_{i+1} \rho_{i} \mu^{-1}) \qquad \qquad \quad \enspace (\text{Using } \eqref{3}) \\
&=(\mu \rho_i s_{i+1} \rho_{i+1} \rho_i\mu^{-1})(\mu \rho_i s_{i+1} \rho_{i+1} \rho_i\mu^{-1})^{-1}
\end{align*}
gives only trivial relations in $PVT_n$. 
\item[]
\item[(iv)] Finally we consider the relations $ (s_i s_j)^2 = 1$ for $|i-j| \geq 2$. If $\mu = 1$, then we get
\begin{align*}
\tau(s_i s_j s_i s_j) &= \gamma(1, s_i) \gamma(\overline{s_i}, s_j)  \gamma(\overline{s_i s_j}, s_i)  \gamma(\overline{s_i s_j s_i}, s_j)\\
&= \gamma(1, s_i) \gamma(\rho_i, s_j)  \gamma(\rho_i \rho_j, s_i)  \gamma(\rho_i \rho_j \rho_i, s_j)\\
&= (s_i \rho_i) (s_j \rho_j) (\rho_i s_i) ( \rho_j s_j)\\
&= \lambda_{i, i+1} \lambda_{j, j+1} \lambda_{i, i+1}^{-1} \lambda_{j, j+1}^{-1}.
\end{align*}
For $\mu \neq 1$, we have
\begin{align}
\nonumber \tau(\mu s_i s_j s_i s_j \mu^{-1}) &= \gamma(\overline{\mu}, s_i) \gamma(\overline{\mu s_i}, s_j)  \gamma(\overline{\mu s_i s_j}, s_i)  \gamma(\overline{ \mu s_i s_j s_i}, s_j)\\
\nonumber &= \gamma(\mu, s_i) \gamma(\mu \rho_i, s_j)  \gamma(\mu \rho_i \rho_j, s_i)  \gamma(\mu \rho_i \rho_j \rho_i, s_j)\\
\nonumber &= (\mu s_i \rho_i \mu^{-1})(\mu s_j \rho_j \mu^{-1}) (\mu \rho_i s_i \mu^{-1}) (\mu \rho_j s_j \mu^{-1})\\
\nonumber &= (\mu s_i \rho_i \mu^{-1})(\mu s_j \rho_j \mu^{-1}) (\mu s_i \rho_i \mu^{-1})^{-1}(\mu s_j \rho_j \mu^{-1})^{-1}\\
&= (\mu \lambda_{i, i+1} \mu^{-1})(\mu \lambda_{j, j+1} \mu^{-1})(\mu \lambda_{i, i+1} \mu^{-1})^{-1} (\mu \lambda_{j, j+1} \mu^{-1})^{-1}.\label{commuting-relations}
\end{align}
\end{itemize}

For $n \ge 4$, we set 
$$\mathcal{T} = \big\{(\lambda_{i,j}^{\epsilon}, \lambda_{k,l}^{\epsilon'})~|~i, j, k, l~\textrm{are distinct integers with}~1\leq i < j\leq n, ~1\leq k < l\leq n~\textrm{and}~\epsilon, \epsilon'\in \{1, -1\} \big\}.$$
If $\rho \in S_n$ and $(\lambda_{i,j}, \lambda_{k,l}) \in \mathcal{T}$, then Remark \ref{rephrasing-action} implies that $ \rho\lambda_{i,j}\rho^{-1} = \lambda_{i',j'}^{\epsilon}$ and $\rho\lambda_{k,l}\rho^{-1} = \lambda_{k',l'}^{\epsilon'}$ for some $\epsilon, \epsilon'\in \{1, -1\}$ and distinct integers $i',j',k', l'$ with $1\leq i' < j'\leq n$ and $1\leq k' < l'\leq n$. Thus, there is an induced diagonal action of $S_n$ on $\mathcal{T}$ given as $$ \rho \cdot (\lambda_{i,j}^{\epsilon},\; \lambda_{k,l}^{\epsilon'}) =  (\rho\lambda_{i,j}^{\epsilon}\rho^{-1},\; \rho\lambda_{k,l}^{\epsilon'}\rho^{-1}).$$ 
We claim that this action is transitive. Since $| \mathcal{T} | =n(n-1)(n-2)(n-3)$, it is enough to show that the stabiliser of some element of $\mathcal{T}$ has $(n-4)!$ elements. In fact, the stabiliser of the element $(\lambda_{1,2},\; \lambda_{n-1,n})$ equals $\langle \rho_3, \rho_4, \ldots, \rho_{n-1}\rangle \cap \langle \rho_1, \rho_2, \ldots, \rho_{n-3}\rangle$ = $\langle \rho_3, \rho_4, \ldots, \rho_{n-3}\rangle \cong S_{n-4}$, which is of order $(n-4)!$. Thus, the action is transitive and the defining relations of $PVT_n$ obtained from \eqref{commuting-relations} are precisely of the form $$\lambda_{i, j}\lambda_{k, l}= \lambda_{k, l}\lambda_{i, j}$$ for distinct integers $i, j, k, l$ with $1\leq i < j\leq n$ and $1\leq k < l\leq n$. This completes the proof of the theorem.
\end{proof}

Recall that, a group is called a {\it right-angled Artin group} if it has a presentation in which the only relations are the commuting relations among the generators. 
One can associate a graph to each right-angled Artin group whose vertices are the Artin generators and there is an edge between two vertices if and only if they commute. We shall discuss this in detail in Section \ref{section-auto-groupo-pvtn}. It is known that a right-angled Artin group is irreducible if and only if its defining graph does not split as a non-trivial simplicial join (see, for example, \cite[Proposition 2.14 and Corollary 2.15]{Behrstock} and \cite[Lemma 5.1]{Koberda}

\begin{corollary}\label{irred-raag}
The pure virtual twin group $PVT_n$ is an irreducible right-angled Artin group for each $n \geq 2$.
\end{corollary}

\begin{proof}
The group $PVT_n$ is a right-angled Artin group follows from Theorem \ref{pvtn-right-angled-artin}, and irreducibility of $PVT_n$ follows from the fact that the complement of its defining graph is connected.
\end{proof}

\begin{corollary}\label{vtn-residually-finite}
The virtual twin group $VT_n$ is residually finite and Hopfian for each $ n\geq 2$.
\end{corollary}

\begin{proof}
 It is well-known that a right-angled Artin group is linear \cite[Corollary 3.6]{HsuWise1999} and that a finitely generated linear group is residually finite \cite{Mal'cev1940}. Thus, $PVT_n$ is linear, and hence residually finite. Since any extension of a residually finite group by a finite group is residually finite, it follows that $VT_n$ is residually finite. The second assertion follows from the fact that every finitely generated residually finite group is Hopfian \cite{Mal'cev1940}.
\end{proof}

The action of $S_n$ on $PVT_n$ (Remark \ref{rephrasing-action}) gives a group homomorphism $\phi:S_n\rightarrow \Aut(PVT_n)$. We conclude the section with the following observation.

\begin{proposition}\label{sn-action-pvtn}
The homomorphism  $\phi:S_n\rightarrow \Aut(PVT_n)$ is injective for each $n\geq 2$.
\end{proposition}

\begin{proof}
The assertion is immediate for the case $n=2$. Assume that $n \ge 3$.
\par
Case 1. $n\geq 3$ and $n\neq 4$. Recall that the only normal subgroups of $S_n$ are $1$, $A_n$ and $S_n$. Note that the automorphisms $\phi(\rho_1)$, $\phi(\rho_2)$ and $\phi(\rho_1\rho_2)$ are all distinct. Thus, the order of $\im(\phi)$ is strictly greater than $2$, and hence  $\Ker(\phi)$ must be trivial.
\par
Case 2. $n=4$. In this case, the normal subgroups of $S_4$ are $1$, $S_4$, $A_4$ and $$K_4= \{1, (1,\;2)(3,\;4), (1,\;3)(2,\;4), (1,\;4)(2,\;3) \},$$ the Klein four group. As in the previous case, $\im(\phi)$ has more than two elements, and hence $\Ker(\phi)$ is either trivial or $K_4$. Note that $(1,\;2)(3,\;4)=\rho_1\rho_3$. Since $\phi(\rho_1\rho_3)(\lambda_{1,2}) = \lambda_{1,2}^{-1} \neq 1$, we have $\rho_1\rho_3\not\in \Ker(\phi)$. Hence, $\Ker(\phi)$ must be trivial in this case as well.
\end{proof}
\medskip

\section{Decomposition of $PVT_n$ as iterated semi-direct product}\label{section-decom-semidirect-product}
Let $i_n: PVT_{n-1} \to PVT_n$ be the natural inclusion obtained by adding a strand to the rightmost side of the diagram of an element of $PVT_{n-1}$. In the reverse direction, we have a well-defined homomorphism $f_n: PVT_n \to PVT_{n-1} $ obtained by removing the $n$-th strand from the diagram of an element of $PVT_n$. Algebraically, $f_n$ is defined on generators of $PVT_n$ by
\begin{equation*}
f_n (\lambda_{i, j})=
\begin{cases}
\lambda_{i, j} & \text{if}\;\; j \neq n,\\
1 & \text{if}\;\;  j = n.
\end{cases}
\end{equation*}
 Furthermore, we have $f_n \circ i_n= \id_{PVT_{n-1}}$, and hence $f_n$ is surjective. For each $n \ge 2$,  let $U_n$ denote $\Ker(f_n)$. Then we have the split short exact sequence
 $$
\begin{tikzcd}
1\arrow{r} & U_n \arrow{r} & PVT_n \arrow{r}{f_n}& PVT_{n-1} \ar[l , "i_n" ', bend left=-33] \arrow{r} & 1,
\end{tikzcd}
$$
that is, $$PVT_n \cong U_n \rtimes PVT_{n-1}.$$

\begin{theorem}\label{semidirect-decomposition-ptn}
For $n \geq 2$, we have 
$$PVT_n \cong U_n \rtimes (U_{n-1} \rtimes (\cdots \rtimes (U_4 \rtimes(U_3 \rtimes U_2))\cdots )),$$
where $U_2 = PVT_2 \cong \mathbb{Z}$ and $U_i = \Ker(f_i : PVT_i \to PVT_{i-1})$ are infinitely generated free groups for $i \geq 3$.
\end{theorem}

\begin{proof} 
It is clear that $U_2 = PVT_2 \cong \mathbb{Z}$. We use the Reidemeister-Schreier method for $n \geq 3$, for which we take the Schreier system to be $PVT_{n-1}$.
Note that for $\mu \in PVT_{n-1}$ and $\lambda_{i,j} \in \mathcal{S}$, we have
\begin{equation*}
\gamma(\mu,\lambda_{i,j}) =
\begin{cases}
1 & \text{if}\;\; j \neq n,\\
\mu \lambda_{i,j} \mu^{-1} & \text{ if } j=n.
\end{cases}
\end{equation*}
This implies that $U_n$ is generated by the set
$$X=\{ \mu \lambda_{i,n} \mu^{-1} ~|~ \mu \in PVT_{n-1} \text{ and } i=1,2, \dots, n-1 \}.$$
Since $PVT_3 \cong F_3$, it follows that $U_3$ is an infinitely generated free group with generators $$\{ \mu \lambda_{1,3} \mu^{-1},~ \mu \lambda_{2,3}\mu^{-1}~|~ \mu \in PVT_2\}.$$
For $n \geq 4$, in $PVT_n$ we have relations of the form
$$\lambda_{i,j}\lambda_{k,l}\lambda_{i,j}^{-1}\lambda_{k,l}^{-1}=1,$$
where $i, j, k, l$ are distinct integers with $i<j$ and $k < l$. First, consider the case when none of $i, j, k, l$ is equal to $n$. Since $\gamma(\mu, \lambda_{i,j})=1$ for $\mu \in PVT_{n-1}$  and $\lambda_{i,j} \in \mathcal{S}$ with $j \neq n$, we have
$$\tau(\mu \lambda_{i,j}\lambda_{k,l}\lambda_{i,j}^{-1}\lambda_{k,l}^{-1} \mu^{-1} )= 1,$$ 
and hence there is no non-trivial relation in this case.
\par
Next, we consider the case when exactly one of $i, j, k, l$ is equal to $n$. Without loss of generality, we can assume that $j=n$. Applying the rewriting process to the relations $\lambda_{i,n}\lambda_{k,l}\lambda_{i,n}^{-1}\lambda_{k,l}^{-1}=1$ of $PVT_n$ gives
\begin{align*}
\tau(\mu \lambda_{i,n}\lambda_{k,l}\lambda_{i,n}^{-1}\lambda_{k,l}^{-1} \mu^{-1} ) &=\gamma(\overline{\mu}, \lambda_{i,n}) \gamma(\overline{\mu \lambda_{i,n}},\lambda_{k,l}) \gamma(\overline{\mu \lambda_{i,n}\lambda_{k,l}},\lambda_{i,n}^{-1}) \gamma(\overline{\mu \lambda_{i,n}\lambda_{k,l} \lambda_{i,n}^{-1}},\lambda_{k,l}^{-1})\\
&= \gamma(\overline{\mu}, \lambda_{i,n})\gamma(\overline{\mu \lambda_{i,n}\lambda_{k,l}},\lambda_{i,n}^{-1})\\
&= \gamma(\mu, \lambda_{i,n})\gamma(\mu \lambda_{k,l},\lambda_{i,n}^{-1})\\
&=\lambda_{i,n}^{\mu^{-1}} \big(\lambda_{i,n}^{\lambda_{k,l}^{-1}\mu^{-1}}\big)^{-1}.
\end{align*}
This gives the relations
$$\mu \lambda_{i,n} \mu^{-1}=  \mu \lambda_{k,l} \lambda_{i,n} \lambda_{k,l}^{-1} \mu^{-1}$$
in $U_n$, which simply identifies two generators of $U_n$. Finally, to prove that there are still infinitely many distinct generators in the set $X$, we consider the sequence of elements
$$\alpha_\epsilon = \lambda_{n-2, n-1}^{-\epsilon} \lambda_{n-1, n} \lambda_{n-2, n-1}^{\epsilon}$$
in $X$, where $\epsilon \in \mathbb{Z}$. Notice that $\alpha_\epsilon \neq \alpha_{\epsilon'}$ for $\epsilon \neq \epsilon'$. Hence, $U_n$ is an infinite rank free group for each $n \ge 3$.
\end{proof}

Theorem \ref{semidirect-decomposition-ptn} is an analogue of a similar result of Bardakov \cite[Theorem 2]{Bardakov2004} for the virtual pure braid group and of Markoff \cite{Markoff1945} for the classical pure braid group.

\begin{corollary}\label{center-of-PVT_n}
$\Z(PVT_n)=1$ for $n \geq 3$ and $\Z(VT_n)=1$ for $n \geq 2$.
\end{corollary}

\begin{proof}
We use the elementary fact that if $G = N\rtimes H$ is an internal semi-direct product with $\Z(H) = 1$, then $\Z(G) \leq \Z(N)$. Recall that $PVT_2\cong \mathbb{Z}$ and $PVT_3 \cong F_3$. Since $PVT_4 = U_4 \rtimes PVT_3$, we have $$\Z(PVT_4) \leq \Z(U_4).$$
By Theorem \ref{semidirect-decomposition-ptn}, $\Z(U_4)=1$, and hence $\Z(PVT_4) = 1$. An easy induction gives $\Z(PVT_n) = 1$ for all  $n \geq 4$.
\par
Since $VT_2 \cong \mathbb{Z}_2 * \mathbb{Z}_2$, we have $\Z(VT_2) = 1$. For $n\geq 3$, since $\Z(S_n)=1$ and $VT_n = PVT_n \rtimes S_n$, we have  $\Z(VT_n) \leq \Z(PVT_n)=1$.
\end{proof}

A right-angled Artin group is called {\it spherical} if its corresponding Coxeter group is finite. $PT_n$ is clearly non-spherical for $n\ge 3$. It is a well-known conjecture that every irreducible non-spherical Artin group has trivial center \cite[Conjecture B]{GodelleParis}. In view of Corollary \ref{irred-raag}, $PVT_n$ is irreducible and non-spherical. Hence, by Corollary \ref{center-of-PVT_n}, $PVT_n$ satisfies the conjecture for $n \geq 3$.
\medskip

\section{Automorphism group of $PVT_n$}\label{section-auto-groupo-pvtn}

\subsection{Graph of $PVT_n$}
Given a graph $\Gamma$ with the vertex set $V$, the right-angled Artin group associated to $\Gamma$ is the group
$$A_{\Gamma} =\big\langle V \mid [v,\; w ]=1\; \text{iff}\; v, w\in V\; \text{are joined by an edge in}\; \Gamma \big\rangle.$$ 
Conversely, each right-angled Artin group gives a graph (corresponding to its Artin presentation) whose vertex set is the set of generators of the group and there is an edge between the two vertices if and only if the two generators commute. It is easy to see that the right-angled Artin group corresponding to the complete graph on $n$ vertices is the free abelian group $\mathbb{Z}^n$ and the group corresponding to the edgeless graph on $n$ vertices is the free group $F_n$.
\par

The generating set $\mathcal{S} = \{\lambda_{i, j} \mid 1\leq i < j \leq n\}$ is the vertex set of the defining graph of $PVT_n$ for $n \ge 2$.  Note that $PVT_2\cong \mathbb{Z}$, $PVT_3\cong F_3$ and $PVT_4\cong (\mathbb{Z}\times \mathbb{Z})\ast (\mathbb{Z}\times \mathbb{Z})\ast (\mathbb{Z} \times \mathbb{Z})$. While the graphs of $PVT_2$ and $PVT_3$ are edgeless graphs on 1 and 3 vertices, respectively, the graphs of $PVT_4$ and $PVT_5$ are shown in figures \ref{graph-pvt4} and \ref{graph-pvt5}, respectively.  It is interesting to note that the graph of $PVT_n$ is the Kneser graph $K(n, 2)$, which is the same as the commuting graph of the conjugacy class of transpositions in $S_n$.  

\begin{figure}[hbtp]
\centering
\includegraphics[scale=0.5]{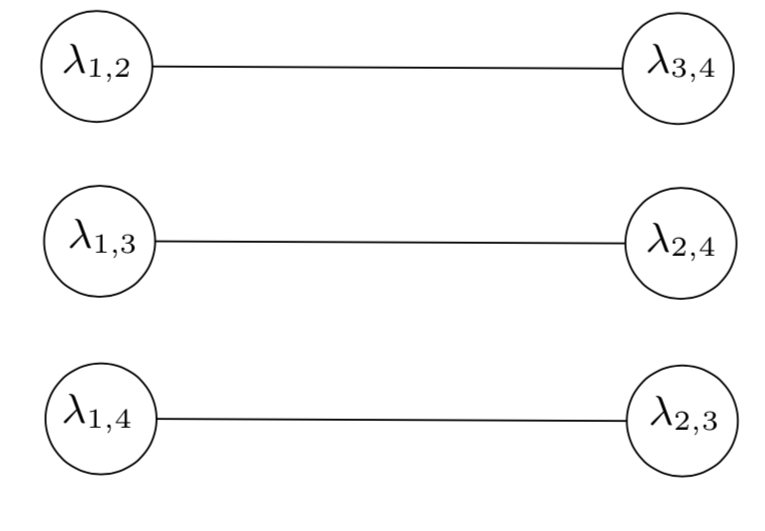}
\caption{The graph of $PVT_4$}
\label{graph-pvt4}
\end{figure}

\begin{figure}[hbtp]
\centering
\includegraphics[scale=0.7]{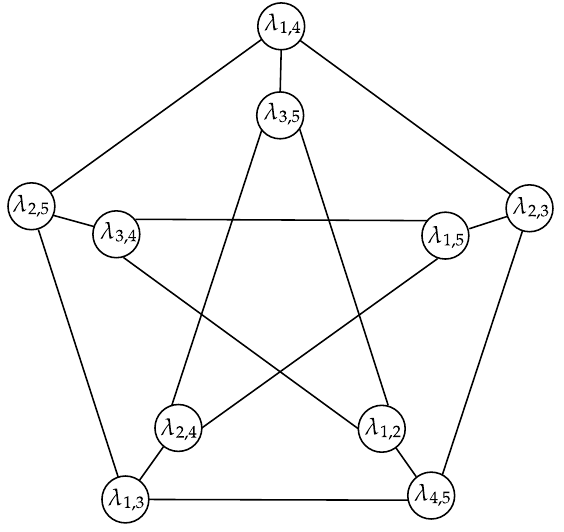}
\caption{The graph of $PVT_5$}
\label{graph-pvt5}
\end{figure}
\medskip

The {\it link} $lk(v)$ of a vertex $v \in V$ is defined as the set of all vertices that are connected to $v$ by an edge. The {\it star}  $st(v)$ of $v$ is defined as $lk(v) \cup \{v\}$. If $v\neq w$ are vertices, then we say that $w$ dominates $v$, written $v \leq w$,  if $lk(v) \subseteq st(w)$.
\par

For each $\lambda_{i, j}\in \mathcal{S}$, let $$N_{i,j}  = \mathcal{S} \setminus st(\lambda_{i,j}) = \big\{\lambda_{k,l}\mid [\lambda_{i,j},\; \lambda_{k,l}] \neq 1\big\}$$
be the set of vertices that are not connected to $\lambda_{i, j}$ by an edge.

\begin{proposition}\label{size-of-non-commuting-generators}
The following hold for each $\lambda_{i,j} \in \mathcal{S}$: 
\begin{enumerate}
\item $| N_{i,j} | = 2n - 4$.
\item $| st(\lambda_{i,j})| = \dfrac{(n-2)(n-4)+n}{2}$. In particular, the graph of $PVT_n$ is  regular. 
\end{enumerate}
\end{proposition}

\begin{proof}
Note that the set $N_{i,j}$ can be written as a union of four disjoint subsets as follows
\begin{align*}
N_{i,j} = &\{ \lambda_{i, k} \mid i+1\leq k \leq n,~ k\neq j\} \cup \{ \lambda_{j, k} \mid j+1\leq k \leq n\}&\\
&\cup \{ \lambda_{k, i} \mid 1\leq k \leq i-1\} \cup \{ \lambda_{k, j} \mid 1\leq k \leq j-1,~ k\neq i\}.&
\end{align*}
Observe that $$| \{ \lambda_{i, k} \mid i+1\leq k \leq n, ~k\neq j\} | + | \{ \lambda_{k, i} \mid 1\leq k \leq i-1\} | =n-2,$$
$$ | \{ \lambda_{j, k} \mid j+1\leq k \leq n\} | + | \{ \lambda_{k, j} \mid 1\leq k \leq j-1, ~k\neq i\}| =n-2,$$
and hence $| N_{i,j} | = 2n - 4$. The second assertion is immediate.
\end{proof}

\subsection{Automorphism group of $PVT_n$ for $n \neq 4$}

It is well-known (see, for example,  \cite{Laurence1995, Servatius1989}) that the automorphism group $\Aut(A_{\Gamma})$ of a right-angled Artin group $A_\Gamma$ is generated by the following four types of automorphisms.

\begin{enumerate}
\item Graph automorphism: Automorphism of $A_{\Gamma}$ induced by an automorphism of the graph $\Gamma$.
\item Inversion $\iota_a$: Sends a generator $a$ to its inverse and leaves all other generators fixed.
\item Transvection $\tau_{ab}$: Sends a generator $a$ to $ab$ and leaves all other generators fixed, where $b$ is another generator with $a \leq b$.
\item Partial conjugation $p_{b, C}$: If $b$ is a generator and $C$ is a union of connected components of $\Gamma \setminus \Gamma (st(b))$, then  $p_{b, C}$ sends each generator $a$ in $C$ to $a^b$ and leaves the other generators fixed. Here, $ \Gamma (st(b))$ is the subgraph of $\Gamma$ spanned by the set $st(b)$. It follows that if $\Gamma \setminus \Gamma (st(b))$ is connected or $C=\Gamma \setminus \Gamma (st(b))$, then the partial conjugation $p_{b, C}$ is simply the inner automorphism induced by $b$.
\end{enumerate}

We consider the following subgroups of the automorphism group $\Aut(PVT_n)$ of $PVT_n$ for $n\geq 2$. 

\begin{itemize}
\item $\Aut_{gr}(PVT_n)$: The subgroup generated by all graph automorphisms.
\item $\Aut_{inv}(PVT_n)$: The subgroup generated by all inversions.
\item $\Aut_{tr}(PVT_4)$: The subgroup generated by all transvections.
\item $\Aut_{pc}(PVT_n)$: The subgroup generated by all partial conjugations.
\item $\Inn(PVT_n)$: The subgroup of all inner automorphisms.
\end{itemize}

Following is an immediate consequence of the definition.

\begin{lemma}\label{inversion-structure}
$\Aut_{inv}(PVT_n) \cong \mathbb{Z}_2^{n(n-1)/2}$ for all $n\geq 2$.
\end{lemma}

Since the graph of $PVT_n$ is the Kneser graph $K(n, 2)$, its group of graph automorphisms is well-known \cite[Corollary 7.8.2]{GodsilRoyle2001}. However, we give a direct computation of $\Aut_{gr}(PVT_n)$ in our set-up. For $n=2$, there is only one vertex, and hence $\Aut_{gr}(PVT_2)=1$. Assume that $n\geq 3$. For each $1\leq k \leq n-1$, define $$\theta_k := \iota_{\lambda_{k,k+1}} \circ \phi(\rho_k),$$
 where $\phi: S_n \to \Aut(PVT_n)$ is the map from Proposition \ref{sn-action-pvtn}. The action of $\theta_k$ on the set $\mathcal{S}$ of generators is described explicitly as follows:

\begin{equation*}
\theta_k :
\begin{cases}
\lambda_{k, k+1} \longrightarrow \lambda_{k, k+1}, & \\
\lambda_{k, j} \longrightarrow \lambda_{k+1, j} & \text{for all}\;\; k+2 \leq j \leq n,\\
\lambda_{k+1, j} \longrightarrow \lambda_{k, j} & \text{for all}\;\; k+2 \leq j \leq n,\\
\lambda_{i, k} \longrightarrow \lambda_{i, k+1} & \text{for all}\;\; 1\leq i < k,\\
\lambda_{i, k+1} \longrightarrow \lambda_{i, k} & \text{for all}\;\;  1\leq i < k,\\
\lambda_{i, j} \longrightarrow \lambda_{i, j} & \text{else, i.e.,}\;\; i \geq k+2,\; \text{or}\;\; i < k \;\text{and}\; k\neq j \neq k+1.
\end{cases}
\end{equation*}

Since each automorphism $\theta_k$ keeps the set $\mathcal{S}$ invariant, we have 
\begin{equation}\label{thetas-in-graph}
\langle \theta_1, \theta_2, \ldots, \theta_{n-1}\rangle \le \Aut_{gr}(PVT_n).
\end{equation}
Note that this action of  $\langle \theta_1, \theta_2, \ldots, \theta_{n-1}\rangle$ on $\mathcal{S}$ is transitive. In fact, given any generator $\lambda_{i, j} \in \mathcal{S}$, the automorphism $(\theta_{i-1}\theta_{i-2} \cdots \theta_{2}\theta_{1})(\theta_{j-1}\theta_{j-2} \cdots \theta_{3}\theta_{2})$ maps $\lambda_{1, 2}$ onto $\lambda_{i, j}$.

\begin{lemma}
$\langle \theta_1, \theta_2, \ldots, \theta_{n-1}\rangle\cong S_n$ for all $n\geq 3$.
\end{lemma}

\begin{proof} Note that  $\theta_i$ is an involution for each $1\leq i\leq n-1$. It follows from the construction  that $[\theta_i,\; \theta_j] = 1$ for all $|i - j| \ge 2$. We now claim that $(\theta_i \theta_{i+1})^3 = 1$ for all $1\leq i \leq n-2$. We verify this for $i=1$ and other cases will follow similarly. Consider
$$
\theta_1:  \left\{
 \begin{array}{ll}
\lambda_{1, 2} \longrightarrow \lambda_{1, 2}, & \\
\lambda_{1, j} \longrightarrow \lambda_{2, j} & \text{for all}\;\; 3 \leq j \leq n,\\
\lambda_{2, j} \longrightarrow \lambda_{1, j} & \text{for all}\;\; 3 \leq j \leq n,\\
\lambda_{i, j} \longrightarrow \lambda_{i, j} & \text{otherwise, i.e.,}\;\; i \geq 3,
\end{array}
 \right.~~~
\theta_2 : \left\{
 \begin{array}{ll}
\lambda_{1, 2} \longrightarrow \lambda_{1, 3}, \\
\lambda_{1, 3} \longrightarrow \lambda_{1, 2},\\
\lambda_{2, 3} \longrightarrow \lambda_{2, 3}, & \\
\lambda_{2, j} \longrightarrow \lambda_{3, j} & \text{for all}\;\; 4 \leq j \leq n,\\
\lambda_{3, j} \longrightarrow \lambda_{2, j} & \text{for all}\;\; 4 \leq j \leq n,\\
\lambda_{i, j} \longrightarrow \lambda_{i, j} & i \geq 4,\; \text{or}\;\; i = 1 \;\text{and}\; 4\leq j \leq n.
\end{array}
\right.
$$

Note that $\theta_1\theta_2$ fixes the generators $\lambda_{i,j}$ for all $i\geq 4$. Thus, we need to verify only for $\lambda_{i,j}$ with $1 \le i\leq 3$. We see that
\begin{equation*}
\theta_1\theta_2 :
\begin{cases}
\lambda_{1, 2} \longrightarrow \lambda_{2, 3}, & \\
\lambda_{1, 3} \longrightarrow \lambda_{1, 2}, & \\
\lambda_{1, j} \longrightarrow \lambda_{2, j} & \text{for all}\;\; 4 \leq j \leq n,\\
\lambda_{2, 3} \longrightarrow \lambda_{1, 3}, & \\
\lambda_{2, j} \longrightarrow \lambda_{3, j} & \text{for all}\;\; 4 \leq j \leq n,\\
\lambda_{3, j} \longrightarrow \lambda_{1, j} & \text{for all}\;\; 4 \leq j \leq n,\\
\lambda_{i, j} \longrightarrow \lambda_{i, j} & \text{otherwise, i.e.,}\;\; i \geq 4,
\end{cases}
\end{equation*}
and it can be easily seen that $(\theta_1\theta_2)^3=1.$ Thus, sending $\rho_k$ to $\theta_k$ gives a surjective homomorphism from $S_n$ onto $\langle \theta_1, \theta_2, \ldots, \theta_{n-1}\rangle$.
We claim that this homomorphism is injective as well. Note that the only normal subgroups of $S_n$ are $1$, $A_n$ and $S_n$, $n\geq 3$ and $n\neq 4$. We know $\rho_1\rho_2\in A_n$. From above computation we have $\theta_1\theta_2\neq 1.$ Thus $\rho_1\rho_2$ does not belong to the kernel and consequently, neither $A_n$ nor $S_n$ can be the kernel. So we are done for the case $n\geq 3$ and $n\neq 4$. For $n=4$, we have an extra normal subgroup of $S_4$, namely the Klein four-subgroup, $K_4= \{1, (1, 2)(3, 4), (1, 3)(2, 4), (1, 4)(2, 3)\}$. Like in the previous case, we can see that $\theta_1\theta_3\neq 1.$ Thus $\rho_1\rho_3=(1, 2)(3, 4)
$ does not belong to the kernel and consequently $K_4$ cannot be the kernel. Thus, the kernel of the homomorphism must be trivial and $\langle \theta_1, \theta_2, \ldots, \theta_{n-1}\rangle \cong S_n$.
\end{proof}

For each fixed $k$, we refer to the set $\{ \lambda_{k, l}\mid k< l\leq n\}$ as the  $k^{\mathrm{th}}$ column of the graph of $PVT_n$.

\begin{lemma}\label{first-row-fixed-implies-identity}
Let $\phi \in \Aut_{gr}(PVT_n)$ such that $\phi(\lambda_{1,j}) = \lambda_{1,j}$ for all $2\leq j\leq n$. Then $\phi$ is the identity automorphism.
\end{lemma}

\begin{proof}
We begin by noting that if $\phi(\lambda_{i,j})=\lambda_{i,j}$ for some $i, j$, then being a graph automorphism, $\phi$ keeps the set $N_{i,j}$ invariant. For $1\leq i \leq n-1$, we have $$(i+1)^{\mathrm{th}}\; \text{column}\; \subseteq N_{i,i+1} \subseteq \bigcup_{k=1}^{i+1}\; (k^{\mathrm{th}}\; \text{column}).$$

We now proceed to the main part of the proof. Since  $\phi(\lambda_{1,2}) = \lambda_{1,2}$, $\phi$ keeps $N_{1, 2}$, i.e., the first two columns invariant. As the first column is already fixed pointwise, $\phi$ keeps the second column invariant. Now suppose that $\phi(\lambda_{2,3}) = \lambda_{2, j}$ for some $4\leq j \leq n$. Now we have two elements $\lambda_{1,3}$ and $\lambda_{2,3}$ which do not commute, but their images ($\lambda_{1,3}$ and $\lambda_{2,j}$, $j\geq 4$ respectively) under $\phi$ commute. This contradicts the hypothesis that $\phi$ is an automorphism. Thus $\phi(\lambda_{2,3}) = \lambda_{2, 3}$, and similarly we can show that $\phi(\lambda_{2,j}) = \lambda_{2, j}$, for all $4\leq j\leq n$.  Now consider the element $\lambda_{2,3}$. As $\phi$ fixes $\lambda_{2,3}$, it should keep $N_{2, 3}$ invariant. But $\phi$ already fixes first two columns pointwise. Thus $\phi$ keeps the third column invariant. By repeated use of above arguments we get the desired result.
\end{proof}

\begin{theorem}\label{group-automorphisms}
If $n \ge 3$ and $n \ne 4$, then $\Aut_{gr}(PVT_n) = \langle \theta_1, \theta_2, \ldots, \theta_{n-1}\rangle\cong S_n$. 
\end{theorem}

\begin{proof}
Since the graph of $PVT_n$ has $n(n-1)/2$ vertices, it follows that
$$S_n\cong \langle \theta_1, \theta_2, \ldots, \theta_{n-1}\rangle \leq \Aut_{gr}(PVT_n) \leq S_{n(n-1)/2}.$$
Thus, we have $\Aut_{gr}(PVT_3)\cong S_3$. For $n \ge 5$, it suffices to prove $\Aut_{gr}(PVT_n) \leq \langle \theta_1, \theta_2, \ldots, \theta_{n-1}\rangle$. Consider an arbitrary automorphism $\phi \in \Aut_{gr}(PVT_n)$. Our plan is to compose $\phi$ with finitely many automorphisms of $\langle \theta_1, \theta_2, \ldots, \theta_{n-1}\rangle$ to obtain the identity automorphism. This would be achieved by invoking Lemma \ref{first-row-fixed-implies-identity}.

As the automorphism group $\langle \theta_1, \theta_2, \ldots, \theta_{n-1}\rangle$ acts transitively on the generating set $\mathcal{S}$, there exists $\phi_1\in \langle \theta_1, \theta_2, \ldots, \theta_{n-1}\rangle$ such that $\phi_1 \phi (\lambda_{1,2}) = \lambda_{1,2}$. Thus, the set $N_{1, 2} = \{ \lambda_{i, j} \mid 1\leq i \leq 2,\; 3\leq j \leq n\}$ should be invariant under $\phi_1 \phi$. Now set

\begin{equation*}
\phi_2 :=
\begin{cases}
1 & \text{if}\;\; \phi_1 \phi (\lambda_{1,3}) = \lambda_{1,3},\\
\theta_1 & \text{if}\;\; \phi_1 \phi (\lambda_{1,3}) = \lambda_{2,3},\\
\theta_3 \theta_4 \cdots \theta_{j-1}    & \text{if}\;\; \phi_1 \phi (\lambda_{1,3}) = \lambda_{1,j}, j\ge 4,\\
\theta_3\theta_4\cdots    \theta_{j-1}\theta_1  & \text{if}\;\; \phi_1 \phi (\lambda_{1,3}) = \lambda_{2,j}, j\ge 4.
\end{cases}
\end{equation*}

Note that $\phi_2 \phi_1 \phi(\lambda_{1,2}) = \lambda_{1,2}$ and $\phi_2 \phi_1 \phi(\lambda_{1,3}) = \lambda_{1,3}$. It follows that $\phi_2 \phi_1 \phi$ keeps the sets $N_{1, 2}$ and $lk(\lambda_{1,3})$ invariant and thus keeps their intersection also invariant. Note that $$N_{1, 2}\cap lk(\lambda_{1,3}) = \big\{\lambda_{2, j}\mid 4\leq j \leq n \big\}.$$

We claim that there exists $\phi_3 \in \langle \theta_4, \theta_5, \ldots, \theta_{n-1}\rangle$ such that $\phi_3\phi_2 \phi_1 \phi(\lambda_{2,j}) = \lambda_{2, j}$ for all $4\leq j\leq n$. Suppose that $\phi_2 \phi_1 \phi(\lambda_{2, 4}) = \lambda_{2, j}$ for some $4\leq j\leq n$. Take $\psi_4=\theta_4\theta_5\ldots \theta_{j-1}\in \langle \theta_4, \theta_5, \ldots, \theta_{n-1}\rangle$. Note that $\psi_4\phi_2 \phi_1 \phi(\lambda_{2, 4}) = \psi_4(\lambda_{2, j}) = \lambda_{2, 4}$ and $\psi_4\phi_2 \phi_1 \phi$ keeps the set $\{\lambda_{2, j} \mid 5\leq j \leq n\}$ invariant. Now suppose that $\phi_2 \phi_1 \phi(\lambda_{2, 5}) = \lambda_{2, j}$ for some $5\leq j\leq n$. Take $\psi_5=\theta_5\theta_6\ldots \theta_{j-1}\in \langle \theta_5, \theta_6, \ldots, \theta_{n-1}\rangle \subseteq \langle \theta_4, \theta_5, \ldots, \theta_{n-1}\rangle$. We repeat the argument and take $\phi_3 = \psi_n\psi_{n-1}\ldots \psi_4 \in \langle \theta_4, \theta_5, \ldots, \theta_{n-1}\rangle$. It follows from the choice of $\phi_3$ that $\phi_3\phi_2 \phi_1 \phi(\lambda_{1, 2}) = \lambda_{1, 2}$ and $\phi_3\phi_2 \phi_1 \phi(\lambda_{1,3}) = \lambda_{1, 3}$.
\par
It now suffices to show that $\phi_3\phi_2 \phi_1 \phi(\lambda_{1, j})=\lambda_{1,j}$ for all $j \ge4$. If $\phi_3\phi_2 \phi_1 \phi(\lambda_{1, j})=\lambda_{2,3}$, then
there exists $k \neq j$ with $k\ge 4$ such that  $\phi_3\phi_2 \phi_1 \phi(\lambda_{1, k})=\lambda_{1,l}$ for some $l \ge 4$. But, this leads to a contradiction since non-commuting generators have commuting images. Hence, $\phi_3\phi_2 \phi_1 \phi(\lambda_{1, j})=\lambda_{1, l}$ for some $l \ge 4$. If $l \neq j$, then we again arrive at a contradiction as before. Thus, $\phi_3\phi_2 \phi_1 \phi$ fixes the first column and we are done by Lemma \ref{first-row-fixed-implies-identity}.
\end{proof}

%We claim that $\phi_3\phi_2 \phi_1 \phi(\lambda_{2,3}) = \lambda_{2, 3}$ and $\phi_3\phi_2 \phi_1 \phi(\lambda_{1, j}) = \lambda_{1, j}$ for all $4\leq j\leq n$. We use the commuting relations among elements of $N_{1, 2}$ and the hypothesis that $n \ge 5$.

\begin{theorem}\label{full-auto-decomposition}
$ \Aut(PVT_n) = \big\langle \Aut_{gr}(PVT_n), \Aut_{inv}(PVT_n), \Inn(PVT_n) \big\rangle$ for all $n\geq 5$.
\end{theorem}

\begin{proof}
Our first claim is that $PVT_n$ does not admit any automorphism of transvection type. Equivalently, if $\lambda_{i, j} \neq \lambda_{k,\;l}$,  then neither $\lambda_{i, j} \leq \lambda_{k,\;l}$ nor $\lambda_{k,\;l} \leq \lambda_{i, j}$. Suppose that $\lambda_{i, j} \neq \lambda_{k,\;l}$. That means either $\{i, j\} \cap \{k, l \}$ is empty or a singleton set. 
Let us first suppose that the intersection is empty. Since $n\geq 5$, we can choose $1\leq q \leq n$ such that $q\notin \{i, j, k, l\}$. Set $x = \lambda_{i, q}$ (if $i< q$) or $x = \lambda_{q, i}$ (if $i> q$) and $y =\lambda_{k, q}$ (if $k< q$) or $y = \lambda_{q, k}$ (if $k> q$). Observe that $x\in lk(\lambda_{k, l}) \setminus st(\lambda_{i, j})$ and $y\in lk(\lambda_{i, j}) \setminus st(\lambda_{k, l})$. Thus, neither $\lambda_{k, l} \leq \lambda_{i, j}$ nor $\lambda_{i, j} \leq \lambda_{k, l}$.
\par
Now we suppose that the intersection is a singleton set. Without loss of generality, we may assume that $i=k$. Since $n \geq 5$, choose $m \notin \{ i, j, l \}$ and set $x = \lambda_{m,l}$ (if $m< l$) or $x = \lambda_{l, m}$ (if $m> l$) and $y =\lambda_{m, j}$ (if $m< j$) or $y = \lambda_{j, m}$ (if $m> j$). Observe that $x\in lk(\lambda_{i, j}) \setminus st(\lambda_{k,l})$ and $y\in lk(\lambda_{k,l}) \setminus st(\lambda_{i,j})$. Thus, $PVT_n$ does not admit any automorphism of transvection type for $n\geq 5$.
\par
Our second claim is that  $\Aut_{pc}(PVT_n)=\Inn(PVT_n)$. Equivalently, the subgraph $\Gamma \setminus \Gamma(st(\lambda_{i, j}))$ is connected  for each $\lambda_{i, j}$. Recall that the action of  $\Aut_{gr}(PVT_n) = \langle \theta_1, \theta_2, \ldots , \theta_{n-1}\rangle$ on $\mathcal{S}$ is transitive: given any generator $\lambda_{i, j} \in \mathcal{S}$, the automorphism $(\theta_{i-1}\theta_{i-2} \cdots \theta_{2}\theta_{1})(\theta_{j-1}\theta_{j-2} \cdots \theta_{3}\theta_{2})$ maps $\lambda_{1, 2}$ onto $\lambda_{i, j}$. Thus, it suffices to prove the claim for $\lambda_{1, 2}$. Note that the vertex set of $\Gamma \setminus \Gamma(st(\lambda_{1, 2}))$ is $\{\lambda_{1, i} , \lambda_{2, j} \mid 3\leq i, j \leq n\}$. Let $v_1, v_2$ be two vertices of $\Gamma \setminus \Gamma(st(\lambda_{1, 2}))$. We find a path joining these two vertices as per the following cases:
\begin{enumerate}
\item $v_1 = \lambda_{1, i}, v_2=\lambda_{1, j}, i\neq j$: Choose an integer $k$ such that $3\leq k\leq n$ and $i\neq k \neq j$. This is possible since $n\geq 5$. We see that there is an edge joining $\lambda_{1, i}$ and $\lambda_{2, k}$ and an edge joining $\lambda_{2, k}$ and $\lambda_{1, j}$.

\item $v_1 = \lambda_{2, i}, v_2=\lambda_{2, j}, i\neq j$: This is analogous to the previous case.

\item $v_1 = \lambda_{1, i}, v_2=\lambda_{2, j}, i\neq j$: Clearly there is an edge joining $\lambda_{1, i}$ and $\lambda_{2, j}$.

\item $v_1 = \lambda_{1, i}, v_2=\lambda_{2, i}$: Choose two integers $j, k$ such that $3\leq j, k \leq n$ and $j\neq i \neq k\neq j$. We can see that there are edges from $\lambda_{1, i}$ to $\lambda_{2, j}$, from $\lambda_{2, j}$ to $\lambda_{1, k}$, and from $\lambda_{1, k}$ to $\lambda_{2, i}$.
\end{enumerate}
Hence, the subgraph $\Gamma \setminus \Gamma(st(\lambda_{1, 2}))$ is connected. Thus, $\Aut_{pc}(PVT_n)=\Inn(PVT_n)$ for $n \ge 5$. Finally, by \cite{Laurence1995, Servatius1989}, we have $\Aut(PVT_n) = \big\langle \Aut_{gr}(PVT_n), \Aut_{inv}(PVT_n), \Inn(PVT_n) \big\rangle$.
\end{proof}

\begin{theorem}\label{aut-pvtn-main}
Let $n \ge 5$. Then there exist split exact sequences 
\begin{equation}\label{auto-seq-1}
1\rightarrow \Aut_{inv}(PVT_n)\rightarrow   \langle \Aut_{gr}(PVT_n), \Aut_{inv}(PVT_n) \rangle \rightarrow \Aut_{gr}(PVT_n) \to 1
\end{equation}
and 
\begin{equation}\label{auto-seq-2}
1\rightarrow \Inn(PVT_n)\rightarrow \Aut(PVT_n)\rightarrow  \langle \Aut_{gr}(PVT_n), \Aut_{inv}(PVT_n) \rangle \to 1.
\end{equation}
In particular,
$$\langle \Aut_{gr}(PVT_n), \Aut_{inv}(PVT_n) \rangle \cong \mathbb{Z}_2^{n(n-1)/2} \rtimes S_n$$
and 
$$\Aut(PVT_n) \cong PVT_n \rtimes (\mathbb{Z}_2^{n(n-1)/2} \rtimes S_n).$$
\end{theorem}

\begin{proof}
It follows from the construction of graph automorphisms that $\Aut_{gr}(PVT_n)$ normalises $\Aut_{inv}(PVT_n)$. Further, since $\Aut_{gr}(PVT_n) \cap \Aut_{inv}(PVT_n)=1$, the short exact sequence \eqref{auto-seq-1} splits.
\par

Recall from Theorem \ref{full-auto-decomposition} that $\Aut(PVT_n) = \langle \Aut_{gr}(PVT_n), \Aut_{inv}(PVT_n), \Inn(PVT_n) \rangle$. Note that any automorphism $\phi \in \langle \Aut_{gr}(PVT_n), \Aut_{inv}(PVT_n) \rangle$ keeps the set $\mathcal{S}\cup \mathcal{S}^{-1}$ invariant. Since two distinct elements of $\mathcal{S}\cup \mathcal{S}^{-1}$ are not conjugates of each other in $PVT_n$, it follows that $$\Inn(PVT_n) \cap  \langle \Aut_{gr}(PVT_n), \Aut_{inv}(PVT_n) \rangle = 1.$$ This gives the split sequence \eqref{auto-seq-2}. The remaining two assertions are immediate by Lemma \ref{inversion-structure}, Theorem \ref{group-automorphisms} and Corollary \ref{center-of-PVT_n}.
\end{proof}

Recall that $PVT_2 \cong \mathbb{Z}$ and $PVT_3 \cong F_3$. Thus, $\Aut(PVT_2) \cong \mathbb{Z}_2$, and the structure of $\Aut(F_3)$ is well-known, see, for example, \cite[Corollary 1]{ArmstrongForrestVogtmann2008}. The case $n=4$ is exotic and will be dealt with separately in Subsection \ref{sec-aut-pvt4}.
\medskip

Given a group $G$ and an automorphism $\phi \in \Aut(G)$, two elements $x, y \in G$ are said to lie in the same $\phi$-twisted conjugacy class if there exists $g \in G$ such that $x = gy\phi(g)^{-1}$.  If the group $G$ has infinitely many $\phi$-twisted conjugacy classes for each $\phi \in \Aut(G)$, then we say that the group $G$ has the $R_{\infty }$-property.
\par

Twisted conjugacy generalises the usual conjugacy and has deep connections with Nielsen fixed-point theory. The topic has attracted a lot of attention in recent years and we refer the reader to \cite{Cox, DekimpeGoncalves2014, Timur3, Fel'shtynTroitsky2015, Timur2, Goncalves2, Goncalves1, MubeenaSankaran2014, Timur4, Timur1} for some recent works on the topic.
\par

It is known that braid groups $B_n$  \cite{Felshtyn}, pure braid groups $P_n$ \cite{DekimpeGoncalvesOcampo} and  twin groups $T_n$ \cite{NaikNandaSingh2} have the $R_{\infty }$-property for all $n \geq 3$.  For Artin groups, it is known due to \cite{Juhasz} that length preserving automorphisms of certain Artin groups whose exponents are more than two have infinitely many twisted conjugacy classes. In a recent work \cite{DekimpeSenden}, several subclasses of right-angled Artin groups including Artin groups with transvection-free graphs, strongly regular graphs and non-complete graphs on at most 7 vertices are shown to have the $R_{\infty }$-property.  

\begin{theorem}\label{pvtn-r-infinity}
 $PVT_n$ has $R_\infty$-property if and only if $n \ge 3$.
\end{theorem}

\begin{proof}
Clearly, $PVT_2$ does not have $R_\infty$-property. Since the graphs of $PVT_3$ and $PVT_4$  are non-complete graphs on at most 7 vertices, it follows from \cite[Theorem 7.1.1]{DekimpeSenden} that these groups have $R_\infty$-property. For $n \ge 5$, Theorem \ref{full-auto-decomposition} gives $\Aut(PVT_n) = \big\langle \Aut_{gr}(PVT_n), \Aut_{inv}(PVT_n), \Inn(PVT_n) \big\rangle$. Since the graph of $PVT_n$ is not complete, the result now  follows from  \cite[Theorem 3.3.3]{DekimpeSenden}.
\end{proof}
\medskip

\subsection{Automorphism  group of $PVT_4$}\label{sec-aut-pvt4}
Recall that $PVT_4 \cong (\mathbb{Z} \times \mathbb{Z}) \ast (\mathbb{Z} \times \mathbb{Z}) \ast (\mathbb{Z} \times \mathbb{Z})$. For $1 \le i \le 3$, let $H_i$ denote the $i$-th free abelian factor in the free product decomposition of $PVT_4$. For simplicity of notation, we set
$$H_i  = \langle x_i,\; y_i \mid [x_i,\; y_i]=1\rangle.$$
Recall from \cite{Laurence1995, Servatius1989} that the automorphism group of a right-angled Artin group is generated by graph automorphisms, inversions, transvections and partial conjugations. By looking at the graph of $PVT_4$ (see Figure \ref{graph-pvt4}) one can easily see that $| \Aut_{gr}(PVT_4) | = 48$. In fact,  $\Aut_{gr}(PVT_4)$ is generated by the following five graph automorphisms:

$$
\sigma_1 : \left\{
\begin{array}{ll}
x_1 \longleftrightarrow y_1, & \\
x_i \longleftrightarrow x_i &~\textrm{if}~ i= 2, 3,\\
y_j \longleftrightarrow y_j &~\textrm{if}~ j= 2, 3,
\end{array}
\right.~~~
\sigma_2 : \left\{
\begin{array}{ll}
x_2 \longleftrightarrow y_2, & \\
x_i \longleftrightarrow x_i &~\textrm{if}~ i= 1, 3,\\
y_j \longleftrightarrow y_j &~\textrm{if}~ j= 1, 3,
\end{array}
\right.~~~
\sigma_3 : \left\{
\begin{array}{ll}
x_3 \longleftrightarrow y_3, & \\
x_i \longleftrightarrow x_i &~\textrm{if}~ i= 1, 2,\\
y_j \longleftrightarrow y_j &~\textrm{if}~ j= 1, 2,
\end{array}
\right.
$$

\begin{equation*}
\psi_1 :
\begin{cases}
x_1 \longleftrightarrow x_2, &\\
y_1 \longleftrightarrow y_2, &\\
x_3 \longleftrightarrow x_3, &\\
y_3 \longleftrightarrow y_3, &
\end{cases}
~\textrm{and}~~
\psi_2 :
\begin{cases}
x_1 \longleftrightarrow x_1, &\\
y_1 \longleftrightarrow y_1, &\\
x_2 \longleftrightarrow x_3, &\\
y_2 \longleftrightarrow y_3. &
\end{cases}
\end{equation*}
\par

Let $\iota_{x_i}$ and $\iota_{y_i}$ denote the inversion automorphisms that invert the generators $x_i$ and $y_i$, respectively, and fix all other generators. Let $\tau_{x_1y_1}$, $\tau_{y_1 x_1}$, $\tau_{x_2y_2}$, $\tau_{y_2x_2}$,  $\tau_{x_3y_3}$ and $\tau_{y_3x_3}$ be the transvection automorphisms that generate $\Aut_{tr}(PVT_4)$.
\par

Let $C_i$ denote the connected component of the graph $\Gamma$ of $PVT_4$ corresponding to the subgroup $H_i$ or equivalently to the vertex set $\{x_i, y_i \}$. Then, for the generator $x_1$  of $PVT_4$, there are three choices for a union $C$ of connected components of $\Gamma \setminus \Gamma (st(x_1))$. Thus, there are 18 partial conjugations that generate $\Aut_{pc}(PVT_4)$. The partial conjugations corresponding to the generator $x_1$ are as follows:

$$
p_{x_1, C_2}:  \left\{
\begin{array}{ll}
x_2\rightarrow x_1^{-1}x_2x_1, & \\
y_2\rightarrow x_1^{-1}y_2x_1, &\\
x_j\rightarrow x_j &~\textrm{if}~ j= 1, 3,\\
y_k\rightarrow y_k &~\textrm{if}~ k= 1, 3,
\end{array}
\right.~~~~~
p_{x_1, C_3}:
\left\{
\begin{array}{ll}
x_3\rightarrow x_1^{-1}x_3x_1, & \\
y_3\rightarrow x_1^{-1}y_3x_1, &\\
x_j\rightarrow x_j &~\textrm{if}~ j= 1, 2,\\
y_k\rightarrow y_k &~\textrm{if}~ k= 1, 2,
\end{array}
\right.
$$
\begin{equation*}
p_{x_1, C_2 \cup C_3}:
\begin{cases}
x_i\rightarrow x_1^{-1}x_ix_1 &~\textrm{if}~ i=2, 3, \\
y_i\rightarrow x_1^{-1}y_ix_1 &~\textrm{if}~ i=2, 3,\\
x_1\rightarrow x_1, & \\
y_1\rightarrow y_1. &
\end{cases}
\end{equation*}

Notice that $p_{x_1, C_2}~ p_{x_1, C_3} = p_{x_1, C_2\cup C_3} = p_{x_1, C_3} ~p_{x_1, C_2}$. Moreover, $p_{x_1, C_2\cup C_3}$  is the inner automorphism induced by $x_1$. By symmetry, the remaining 15 generating partial conjugations can be defined analogously, and we have
$$\Aut_{pc}(PVT_4) = \big \langle \Inn(PVT_4), \langle p_{x_1, C_2}, p_{y_1, C_2}, p_{x_2, C_3},p_{y_2, C_3}, p_{x_3, C_1}, p_{y_3, C_1} \rangle \big \rangle.$$
Since $\Aut(H_i)\cong \Aut(\mathbb{Z} \times \mathbb{Z}) \cong \GL_2 (\mathbb{Z})$, it follows that $\Aut(PVT_4)$ contains a subgroup isomorphic to $\GL_2 (\mathbb{Z}) \times \GL_2 (\mathbb{Z}) \times \GL_2 (\mathbb{Z})$.

\begin{lemma}\label{inversion-transvection-automorphism}
$$\langle \Aut_{inv}(PVT_4), \Aut_{tr}(PVT_4) \rangle \cong \GL_2 (\mathbb{Z}) \times \GL_2 (\mathbb{Z}) \times \GL_2 (\mathbb{Z}).$$
\end{lemma}

\begin{proof}
Recall that $$\langle \Aut_{inv}(PVT_4), \Aut_{tr}(PVT_4) \rangle  = \langle \iota_{x_1}, \iota_{x_2}, \iota_{x_3}, \iota_{y_1}, \iota_{y_2}, \iota_{y_3}, \tau_{x_1y_1}, \tau_{y_1x_1}, \tau_{x_2y_2}, \tau_{y_2x_2}, \tau_{x_3y_3}, \tau_{y_3x_3}\rangle.$$ 
Let us set
\begin{align*}
K_1  &= \langle \tau_{x_1y_1}, \tau_{y_1x_1}, \iota_{x_1}, \iota_{y_1} \rangle,&\\
K_2 & = \langle \tau_{x_2y_2}, \tau_{y_2x_2}, \iota_{x_2}, \iota_{y_2} \rangle,&\\
K_3 & = \langle \tau_{x_3y_3}, \tau_{y_3x_3}, \iota_{x_3}, \iota_{y_3}, \rangle.&
\end{align*}
Notice that $K_1, K_2, K_3$ act trivially on $H_2\ast H_3$, $H_1\ast H_3$ and $H_1\ast H_2$, respectively. Further, $[K_1,\; K_2] = [K_2,\; K_3] = [K_3,\; K_1] = 1$ and $K_1 \cong K_2 \cong K_3$. Thus, it suffices to prove that $K_1\cong \GL_2 (\mathbb{Z})$. But, this follows by recalling that
$$\GL_2 (\mathbb{Z})= \Big\langle
\begin{bmatrix}
1 & 1 \\
0 & 1 
\end{bmatrix},
\begin{bmatrix}
1 & 0 \\
1 & 1 
\end{bmatrix}, 
\begin{bmatrix}
-1 & 0 \\
0 & 1 
\end{bmatrix}
\Big\rangle
$$
and identifying generators of $K_1$ with matrices in $\GL_2 (\mathbb{Z})$ as
$$\tau_{x_1y_1} \rightarrow
\begin{bmatrix}
1 & 0 \\
1 & 1 
\end{bmatrix},
\tau_{y_1x_1}\rightarrow
\begin{bmatrix}
1 &1 \\
0 & 1 
\end{bmatrix},
\iota_{x_1} \rightarrow
\begin{bmatrix}
-1 & 0 \\
0 & 1 
\end{bmatrix}
~\textrm{and}~
\iota_{y_1} \rightarrow
\begin{bmatrix}
1 & 0 \\
0 & -1 
\end{bmatrix}
.$$
\end{proof}

\begin{lemma}\label{graph-inv-tr-structure}
$$\Aut_{gr}(PVT_4) \cap \langle \Aut_{inv}(PVT_4), \Aut_{tr}(PVT_4) \rangle = \langle \sigma_1, \sigma_2, \sigma_3 \rangle \cong \mathbb{Z}_2 \times \mathbb{Z}_2 \times \mathbb{Z}_2$$
and
\begin{align*}
\langle \Aut_{gr}(PVT_4), \Aut_{inv}(PVT_4), \Aut_{tr}(PVT_4) \rangle &= \langle \Aut_{inv}(PVT_4), \Aut_{tr}(PVT_4) \rangle \rtimes \langle \psi_1, \psi_2 \rangle &\\
& \cong \big ( \GL_2 (\mathbb{Z}) \times \GL_2 (\mathbb{Z}) \times \GL_2 (\mathbb{Z}) \big )\rtimes S_3.&
\end{align*}
\end{lemma}

\begin{proof}
Notice that $\psi_1$ and $\psi_2$ permute the $H_i$'s non-trivially, and hence $\langle \Aut_{inv}(PVT_4), \Aut_{tr}(PVT_4) \rangle \cap \langle \psi_1, \psi_2 \rangle = 1$. On the other hand $\sigma_1, \sigma_2, \sigma_3 \in \langle \Aut_{inv}(PVT_4), \Aut_{tr}(PVT_4) \rangle$, and hence the first assertion follows.
\par
For the second assertion notice that $\langle \psi_1, \psi_2 \rangle \cong S_3$ and $\psi_1^2=1=\psi_2^2$. It suffices to show that $\langle\psi_1, \psi_2 \rangle$ normalises $\langle \Aut_{inv}(PVT_4), \Aut_{tr}(PVT_4) \rangle$. A direct computation shows that\\
$$
\begin{array}{rclr}
\psi_1 \tau_{x_1y_1} \psi_1 = \tau_{x_2y_2},& \psi_1 \tau_{y_1x_1} \psi_1 = \tau_{y_2x_2},& \psi_1 \tau_{x_3y_3} \psi_1 = \tau_{x_3y_3}, & \psi_1 \tau_{y_3x_3} \psi_1 = \tau_{y_3x_3},\\
\psi_1 \iota_{x_1} \psi_1 = \iota_{x_2}, & \psi_1 \iota_{y_1} \psi_1 = \iota_{y_2}, & \psi_1 \iota_{x_3} \psi_1 =  \iota_{x_3},& \psi_1 \iota_{y_3} \psi_1 = \iota_{y_3},
\end{array}
$$
\\
and hence $\psi_1$ normalises $\langle \Aut_{inv}(PVT_4), \Aut_{tr}(PVT_4) \rangle$. By symmetry, the same assertion holds for $\psi_2$, and we get the desired result.
\end{proof}

\begin{lemma}\label{partialconjuation-is-normal}
$\Aut_{pc}(PVT_4)$ is normal in $\Aut(PVT_4)$.
\end{lemma}

\begin{proof}
Note that $\Inn(PVT_4)\leq \Aut_{pc}(PVT_4)$.  Set $M=\langle p_{x_1, C_2}, p_{y_1, C_2}, p_{x_2, C_3}, p_{y_2, C_3}, p_{x_3, C_1}, p_{y_3, C_1} \rangle$. It suffices to show that $\phi^{-1} M \phi \leq \Aut_{pc}(PVT_4)$ for all $\phi \in \langle \Aut_{gr}(PVT_4), \Aut_{inv}(PVT_4), \Aut_{tr}(PVT_4) \rangle$.
\par
If $\phi = \psi_1$, then\\
$$
\begin{array}{rcl}
\psi_1 p_{x_1, C_2} \psi_1 = p_{x_2, C_1}, & \psi_1 p_{y_1, C_2} \psi_1 = p_{y_2, C_1}, & \psi_1 p_{x_2, C_3}  \psi_1 = p_{x_1, C_3},\\
\psi_1 p_{y_2, C_3}  \psi_1 =  p_{y_1, C_3}, & \psi_1  p_{x_3, C_1}  \psi_1 =  p_{x_3, C_2}, & \psi_1  p_{y_3, C_1}  \psi_1 =  p_{y_3, C_2}.
\end{array}
$$
\\
Thus, $\psi_1$ normalises $\Aut_{pc}(PVT_4)$. By symmetry, $\psi_2$ and hence $\langle \psi_1, \psi_2 \rangle$ normalises $\Aut_{pc}(PVT_4)$. By Lemma \ref{graph-inv-tr-structure}, it remains to show that $\langle \Aut_{inv}(PVT_4), \Aut_{tr}(PVT_4) \rangle$ normalises $\Aut_{pc}(PVT_4)$.
\par

If $\phi =  \iota_{x_1}$, then\\
$$
\begin{array}{rcl}
\iota_{x_1} p_{x_1, C_2}  \iota_{x_1} = p_{x_1, C_2}^{-1} , & \iota_{x_1} p_{y_1, C_2}   \iota_{x_1} = p_{y_1, C_2}, & \iota_{x_1} p_{x_2, C_3} \iota_{x_1} = p_{x_2, C_3},\\
\iota_{x_1} p_{y_2, C_3} \iota_{x_1} = p_{y_2, C_3}, & \iota_{x_1} p_{x_3, C_1}  \iota_{x_1} = p_{x_3, C_1}, & \iota_{x_1}  p_{y_3, C_1} \iota_{x_1} = p_{y_3, C_1}.
\end{array}
$$
\\

Thus, $\iota_{x_1}$ normalises $\Aut_{pc}(PVT_4)$. By symmetry, all the other inversions also normalise $\Aut_{pc}(PVT_4)$, and consequently $\Aut_{inv}(PVT_4)$ normalises $\Aut_{pc}(PVT_4)$.
\medskip

If $\phi = \tau_{x_1y_1}$, then
$$
\begin{array}{rcl}
\tau_{x_1y_1}^{-1} p_{x_1, C_2} \tau_{x_1y_1} = p_{y_1, C_2}^{-1} p_{x_1, C_2}, & \tau_{x_1y_1}^{-1} p_{y_1, C_2} \tau_{x_1y_1} = p_{y_1, C_2}, & \tau_{x_1y_1}^{-1} p_{x_2, C_3} \tau_{x_1y_1} = p_{x_2, C_3},\\
\tau_{x_1y_1}^{-1} p_{y_2, C_3}  \tau_{x_1y_1} = p_{y_2, C_3}, & \tau_{x_1y_1}^{-1} p_{x_3, C_1}  \tau_{x_1y_1} = p_{x_3, C_1}, & \tau_{x_1y_1}^{-1} p_{y_3, C_1}  \tau_{x_1y_1} = p_{y_3, C_1}.
\end{array}
$$
\medskip
Thus, $\tau_{x_1y_1}$ normalises $\Aut_{pc}(PVT_4)$. Similarly, one can show that all other transvections also normalise $\Aut_{pc}(PVT_4)$, which completes the proof of the lemma.
\end{proof}

Finally, we determine the structure of $\Aut_{pc}(PVT_4)$. A presentation of the group of partial conjugations of a right-angled Artin group has been constructed in \cite{Toinet}. 

\begin{lemma}\label{relations in Aut_pc(PVT_4)}
The group $\Aut_{pc}(PVT_4)$ has a presentation with generating set $$\{p_{x_i, C_j}, p_{y_i, C_j}~|~ i \neq j \text{ and } i, j =1, 2, 3\}$$ and following defining relations: 

\begin{enumerate}
\item $[p_{x_i, C_j},\; p_{x_i, C_k}] =[p_{y_i, C_j},\; p_{y_i, C_k}]= [p_{x_i, C_j},\; p_{y_i, C_k}] = [p_{x_i, C_j},\; p_{y_i, C_j}] =1$ for $i = 1, 2, 3$ with $i \neq j \neq k \neq i$.
\item $[p_{x_i, C_j} p_{x_i, C_k},\; p_{x_j, C_k}] = [p_{y_i, C_j} p_{y_i, C_k},\; p_{y_j, C_k}] = [p_{x_i, C_j} p_{x_i, C_k},\; p_{y_j, C_k}] = [p_{y_i, C_j} p_{y_i, C_k},\; p_{x_j, C_k}] =1$~~ for $i, j, k = 1, 2, 3$ with $i \neq j \neq k \neq i$.
\end{enumerate}
In particular, $$\Aut_{pc}(PVT_4) \cong (\mathbb{Z}^2 \ast \mathbb{Z}^2 \ast \mathbb{Z}^2) \rtimes (\mathbb{Z}^2 \ast \mathbb{Z}^2 \ast \mathbb{Z}^2). $$
\end{lemma}

\begin{proof}
Relations in (1) and (2) follow by direct computations together with \cite[Theorem 3.1]{Toinet}. Note that
$$ \Inn(PVT_4) = \big\langle p_{x_i, C_j}p_{x_i, C_k}, p_{y_i, C_j}p_{y_i, C_k} ~|~ i \neq j \neq k \neq i ,~j < k \text{ and } i,j,k=1,2,3 \big\rangle.$$
Setting $$\Aut_{pc\setminus inn}(PVT_4) = \big\langle p_{x_1, C_2}, p_{y_1, C_2}, p_{x_2, C_3}, p_{y_2, C_3}, p_{x_3, C_1}, p_{y_3, C_1} \big\rangle,$$
we see that
\begin{itemize}
\item $\Aut_{pc\setminus inn}(PVT_4) \cong \mathbb{Z}^2 \ast \mathbb{Z}^2 \ast \mathbb{Z}^2.$
\item $\Aut_{pc}(PVT_4) =\Inn(PVT_4) \Aut_{pc\setminus inn}(PVT_4)$.
\end{itemize}
Consider the surjective homomorphism $g: \Aut_{pc}(PVT_4) \to \Aut_{pc\setminus inn}(PVT_4)$ defined on generators as
$$
g: \left\{
        \begin{array}{ll}
      p_{x_1, C_2} \mapsto p_{x_1, C_2} ,  \\
      p_{x_1, C_3} \mapsto p_{x_1, C_2}^{-1},\\
      p_{y_1, C_2} \mapsto p_{y_1, C_2}, \\
      p_{y_1, C_3} \mapsto p_{y_1, C_2}^{-1},\\
        \end{array}
    \right.~~~
g: \left\{
        \begin{array}{ll}
      p_{x_2, C_1} \mapsto p_{x_2, C_3}^{-1} ,  \\
      p_{x_2, C_3} \mapsto p_{x_2, C_3},\\
      p_{y_2, C_1} \mapsto p_{y_2, C_3}^{-1}, \\
      p_{y_2, C_3} \mapsto p_{y_2, C_3},\\
        \end{array}
    \right.~~~
g: \left\{
        \begin{array}{ll}
      p_{x_3, C_1} \mapsto p_{x_3, C_1} ,  \\
      p_{x_3, C_2} \mapsto p_{x_3, C_1}^{-1},\\
      p_{y_3, C_1} \mapsto p_{y_3, C_1}, \\
      p_{y_3, C_2} \mapsto p_{y_3, C_1}^{-1}.\\
        \end{array}
    \right.
$$
\medskip

Note that $\Inn(PVT_4) \subseteq \Ker(g)$. Let $w \in \Ker(g)$ and write $w= xy$ for some $x \in \Inn(PVT_4)$ and $y \in \Aut_{pc\setminus inn}(PVT_4)$. Then we have
$$1 = g(w) =  g(x)g(y) =  y,$$
and hence $\Ker(g)=\Inn(PVT_4)$. This implies that
\begin{equation}\label{pc-inn-pvt4}
\Aut_{pc}(PVT_4) =\Inn(PVT_4)\rtimes \Aut_{pc\setminus inn}(PVT_4),
\end{equation}
and hence 
$$\Aut_{pc}(PVT_4) \cong (\mathbb{Z}^2 \ast \mathbb{Z}^2 \ast \mathbb{Z}^2) \rtimes (\mathbb{Z}^2 \ast \mathbb{Z}^2 \ast \mathbb{Z}^2).$$
This completes the proof.
\end{proof}
Combining the preceding lemmas yield the following theorem.

\begin{theorem}\label{splitting-auto-pvt4}
There exists a split exact sequence
$$1 \to  \Aut_{pc}(PVT_4) \to \Aut(PVT_4)  \to  \big\langle \Aut_{gr}(PVT_4), \Aut_{inv}(PVT_4), \Aut_{tr}(PVT_4)  \big\rangle \to 1.$$
In particular, 
\begin{align*}
\Aut(PVT_4) & = \Aut_{pc}(PVT_4) \rtimes \big\langle \Aut_{gr}(PVT_4), \Aut_{inv}(PVT_4), \Aut_{tr}(PVT_4)  \big\rangle&\\
&\cong ((\mathbb{Z}^2 \ast \mathbb{Z}^2 \ast \mathbb{Z}^2) \rtimes (\mathbb{Z}^2 \ast \mathbb{Z}^2 \ast \mathbb{Z}^2)) \rtimes \big ( (\GL_2 (\mathbb{Z}) \times \GL_2 (\mathbb{Z}) \times \GL_2 (\mathbb{Z})) \rtimes S_3 \big ).&
\end{align*}
\end{theorem}

\begin{proof}
Note that each automorphism in $\Aut_{pc}(PVT_4)$ preserves conjugacy classes of generators. But, the only automorphism in $\langle \Aut_{gr}, \Aut_{inv}, \Aut_{tr} \rangle$ which preserves conjugacy classes of all the generators is the identity automorphism by Lemma \ref{graph-inv-tr-structure}. Hence $$\Aut_{pc}(PVT_4)\cap \big\langle \Aut_{gr}(PVT_4), \Aut_{inv}(PVT_4), \Aut_{tr}(PVT_4) \big\rangle = 1,$$
and the assertion follows.
\end{proof}
\medskip 

An automorphism of a group is called an $\IA$ automorphism if it acts as identity on the abelianisation of the group. Note that inner automorphisms are $\IA$ automorphisms.

\begin{corollary}\label{cor-ia-auto}
Each $\IA$ automorphism of $PVT_n$ is inner if and only if $n=2$ or $n \ge 5$.
\end{corollary}

\begin{proof}
Note that the $\IA$ automorphism group of $PVT_2$ is obviously trivial. Magnus \cite{Magnus1935} gave generators of the group of $\IA$ automorphisms of $F_3\cong PVT_3$ and showed that it contains non-inner automorphisms. Clearly, $\Aut_{pc}(PVT_4)$ is a subgroup of the group of $\IA$ automorphisms of $PVT_4$. It follows from \eqref{pc-inn-pvt4} that $\Aut_{pc}(PVT_4)$ contains non-inner automorphisms as well. For $n \ge 5$, a direct computation using the description of $\Aut(PVT_n)$ in Theorem \ref{aut-pvtn-main} shows that the only $\IA$ automorphisms of $PVT_n$ are the inner automorphisms.
\end{proof} 
\medskip
\medskip

\begin{ack}
Tushar Kanta Naik acknowledges support from the NBHM via grant number 0204/3/2020/R\&D-II/2475. Neha Nanda acknowledges support from the Winning Normandy Postdoctoral Fellowship. Mahender Singh is supported by the Swarna Jayanti Fellowship grants DST/SJF/MSA-02/2018-19 and SB/SJF/2019-20.
\end{ack}


\begin{thebibliography}{1}
\bibitem{ArmstrongForrestVogtmann2008}  Heather Armstrong, Bradley Forrest and Karen Vogtmann,   {\em A presentation for $\Aut(F_n)$},  J. Group Theory 11 (2008), no. 2, 267--276.

\bibitem{Bardakov2004}  Valeriy G. Bardakov,  {\em The virtual and universal braids}, Fund. Math. 184 (2004), 1--18.
 
\bibitem{BarVesSin} Valeriy Bardakov, Mahender Singh and Andrei  Vesnin, {\em Structural aspects of twin and pure twin groups}, Geom. Dedicata 203 (2019), 135--154.


\bibitem{BartholomewFennKamada2019} Andrew Bartholomew, Roger Fenn, Naoko Kamada and Seiichi  Kamada, {\em Colorings and doubled colorings of virtual doodles}, Topology  Appl. 264 (2019),  290--299.

\bibitem{BartholomewFennKamada2018} Andrew Bartholomew, Roger Fenn, Naoko Kamada and Seiichi  Kamada, {\em Doodles on surfaces}, J. Knot Theory Ramifications 27 (2018), no. 12, 1850071, 26 pp.

\bibitem{BartholomewFennKamada2018-2} Andrew Bartholomew, Roger Fenn, Naoko Kamada and Seiichi  Kamada, {\em On Gauss codes of virtual doodles}, J. Knot Theory Ramifications 27 (2018), no. 11, 1843013, 26 pp.

\bibitem{Behrstock} J. Behrstock and R. Charney, {\em Divergence and quasimorphisms of right-angled Artin groups}, Math. Ann. 352(2) (2012), 339--356.


\bibitem{CisnerosFloresJuyumayaMarquez} Bruno Cisneros, Marcelo Flores, Jes\'us Juyumaya and Christopher Roque-M\'{a}rquez,  {\em  An Alexander type invariant for doodles}, (2020), arXiv:2005.06290.

\bibitem{Cox} Charles Garnet Cox, \textit{Twisted conjugacy in Houghton's groups}, J. Algebra 490 (2017), 390--436.

\bibitem{DekimpeSenden} Karel Dekimpe and Pieter Senden, {\em The $R_{\infty}$-property for right-angled Artin groups}, Topology Appl. 293 (2021), 107557.


\bibitem{DekimpeGoncalves2014} Karel Dekimpe and Daciberg Gon\c{c}alves, {\em $R_{\infty}$-property for free groups, free nilpotent groups and free solvable groups}, Bull. Lond. Math. Soc. 46 (2014), no. 4, 737--746.

\bibitem{DekimpeGoncalvesOcampo} Karel Dekimpe, Daciberg Lima Gon\c{c}alves and Oscar Ocampo, {\em The $R_{\infty}$-property for pure Artin braid groups}, Monatsh. Math.  195 (2021), 15--33.

\bibitem{Felshtyn} Alexander Fel'shtyn and Daciberg L. Gon\c{c}alves, \textit{Twisted conjugacy classes in symplectic groups, mapping class groups and braid groups}, Geom. Dedicata 146 (2010), 211--223.

\bibitem{Timur3} Alexander Fel'shtyn and Timur Nasybullov, \textit{The $R_{\infty }$ and $S_{\infty }$ properties for linear algebraic groups}, J. Group Theory 19 (2016), no. 5, 901--921.


\bibitem{Fel'shtynTroitsky2015}  Alexander Fel'shtyn and  Evgenij Troitsky,  \textit{Aspects of the property $R_{\infty }$}, J. Group Theory 18 (2015), no. 6, 1021--1034. 

\bibitem{FennTaylor} Roger Fenn and Paul Taylor, {\em Introducing doodles}, Topology of low-dimensional manifolds (Proc. Second Sussex Conf., Chelwood Gate, 1977), pp. 37--43, Lecture Notes in Math., 722, Springer, Berlin, 1979.

\bibitem{GodelleParis} Eddy Godelle and Luis Paris, {\em Basic questions on Artin-Tits groups}, Configuration spaces, 299--311, CRM Series, 14, Ed. Norm., Pisa, 2012.


\bibitem{GodsilRoyle2001} Chris Godsil and Gordon Royle,  {\em Algebraic graph theory}, Graduate Texts in Mathematics, 207. Springer-Verlag, New York, 2001. xx+439 pp.


\bibitem{Timur2} Daciberg Lima Gon\c{c}alves and Timur Nasybullov, \textit{On groups where the twisted conjugacy class of the unit element is a subgroup}, Comm. Algebra 47 (2019), no. 3, 930--944. 

\bibitem{Goncalves2}
Daciberg Gon\c{c}alves and Parameswaran Sankaran, \textit{Sigma theory and twisted conjugacy, II: Houghton groups and pure symmetric automorphism groups}, Pacific J. Math. 280 (2016), no. 2, 349--369.

\bibitem{Goncalves1}
Daciberg Lima Gon\c{c}alves and Parameswaran Sankaran, \textit{Twisted conjugacy in PL-homeomorphism groups of the circle}, Geom. Dedicata 202 (2019), 311--320. 

\bibitem{GonGutiRoq} Jes\'us Gonz\'alez,  Jos\'{e} Luis Le\'on-Medina and Christopher  Roque, {\em Linear motion planning with controlled collisions and pure planar braids}, Homology Homotopy Appl. 23 (2021), no. 1, 275--296.

\bibitem{Gotin} Konstantin Gotin, {\em Markov theorem for doodles on two-sphere}, (2018), arXiv:1807.05337.


\bibitem{HarshmanKnapp}  N. L. Harshman and A. C. Knapp, {\em Anyons from three-body hard-core interactions in one dimension}, Ann. Physics 412 (2020), 168003, 18 pp.

\bibitem{HsuWise1999} Tim Hsu and Daniel T. Wise, {\em On linear and residual properties of graph products}, Michigan Math. J. 46(2), (1999), 251--259.

\bibitem{Juhasz} A. Juh\'asz, {\em Twisted conjugacy in certain Artin groups}, Ischia group theory 2010, 175--195, World Sci. Publ., Hackensack, NJ, 2012.

 \bibitem{Kauffman1999}  Louis H. Kauffman,  {\em Virtual knot theory}, European J. Combin. 20 (1999), no. 7, 663--690. 
 
 
\bibitem{KauffmanLambropoulou2004} Louis H. Kauffman and Sofia Lambropoulou, {\em Virtual braids}, Fund. Math. 184 (2004), 159--186.

\bibitem{Khovanov} Mikhail Khovanov, {\em Doodle groups}, Trans. Amer. Math. Soc. 349 (1997), 2297--2315.


\bibitem{Koberda} T. Koberda, {\em Right-angled Artin groups and their subgroups}, https://users.math.yale.edu/users/koberda/raagcourse.pdf.


\bibitem{Laurence1995} Michael R. Laurence,  {\em A generating set for the automorphism group of a graph group}, J. London Math. Society 52 (2) (1995), 318--334.

\bibitem{Magnus1935}  Wilhelm Magnus,  {\em  \"{U}ber $n$-dimensionale gittertransformationen}, Acta Math. 64 (1935), no. 1, 353--367.

\bibitem{Magnus1966}   Wilhelm Magnus, Abraham Karrass and  Donald  Solitar, {\em Combinatorial group theory, Presentations of groups in terms of generators and relations}, Interscience Publishers, New York-London-Sydney 1966 xii + 444 pp.

\bibitem{Mal'cev1940}   A. I. Mal'cev, {\em On isomorphic matrix representations of infinite groups of matrices (Russian)}, Mat. Sb. 8 (1940), 405--422 \& Amer. Math. Soc. Transl. (2) 45 (1965), 1--18.

\bibitem{Markoff1945}    A. A. Markoff, {\em Foundations of the algebraic theory of braids}, Trudy Mat. Inst. Steklova, No. 16 (1945), 1--54.

\bibitem{Mostovoy} Jacob Mostovoy, {\em  A presentation for the planar pure braid group}, (2020), arXiv:2006.08007.

\bibitem{MostRoq} Jacob Mostovoy and Christopher Roque-M\'arquez, {\em Planar pure braids on six strands}, J. Knot Theory Ramifications  29  (2020), No. 01, 1950097.

\bibitem{MubeenaSankaran2014} T. Mubeena and P. Sankaran, {\em Twisted conjugacy classes in abelian extensions of certain linear groups}, Canad. Math. Bull. 57 (2014), no. 1, 132--140.


\bibitem{NaikNandaSingh1} Tushar Kanta Naik, Neha Nanda and Mahender Singh, {\em Conjugacy classes and automorphisms of twin groups},  Forum Math. 32 (2020), no. 5, 1095--1108.

\bibitem{NaikNandaSingh2} Tushar Kanta Naik, Neha Nanda and Mahender Singh, {\em Some remarks on twin groups},  J. Knot Theory Ramifications 29 (2020), no. 10, 2042006, 14 pp.

\bibitem{NandaSingh2020} Neha Nanda and Mahender Singh,  {\em Alexander and Markov theorems for virtual doodles}, New York J. Math. 27 (2021), 272--295.


\bibitem{Timur4} Timur Nasybullov, \textit{Reidemeister spectrum of special and general linear groups over some fields contains 1}, J. Algebra Appl. 18 (2019), no. 8, 1950153, 12 pp.

\bibitem{Timur1} Timur Nasybullov, \textit{Twisted conjugacy classes in unitriangular groups}, J. Group Theory 22 (2019), no. 2, 253--266.

\bibitem{Toinet} Emmanuel Toinet, {\em A finitely presented subgroup of the automorphism group of a right-angled Artin group}, J. Group Theory 15 (2012), no. 6, 811--822.

\bibitem{Servatius1989} H. Servatius, {\em Automorphisms of graph groups}, J. Algebra 126 (1989), 34--60.

\bibitem{ShabatVoevodsky} G. B. Shabat and V. A. Voevodsky, {\em Drawing curves over number fields}, The Grothendieck Festschrift, Vol. III, 199--227, Progr. Math., 88, Birkh\"{a}user Boston, Boston, MA, 1990.
\end{thebibliography}
\end{document}